\documentclass[reqno,12pt]{amsart}
\usepackage{amsfonts}
\usepackage{bbm}
\usepackage{}
\numberwithin{equation}{section}
\usepackage{hyperref}
\usepackage{color}
\usepackage{rotating}
\usepackage{indentfirst}
\usepackage{color}
\usepackage{amssymb}
\usepackage{mathrsfs}
\usepackage{xy}
\usepackage{rotating}
\usepackage{tikz}

\xyoption{all}

  \def\Id{\mbox{\rm Id}\,}

\theoremstyle{plain}
\newtheorem{theorem}{\bf Theorem}[section]
\newtheorem{lemma}[theorem]{\bf Lemma}
\newtheorem{corollary}[theorem]{\bf Corollary}
\newtheorem{proposition}[theorem]{\bf Proposition}

\theoremstyle{definition}
\newtheorem{definition}[theorem]{\bf Definition}
\newtheorem{remark}[theorem]{\bf Remark}
\newtheorem{example}[theorem]{\bf Example}

\newcommand{\bt}{\begin{theorem}}
\newcommand{\et}{\end{theorem}}
\newcommand{\bl}{\begin{lemma}}
\newcommand{\el}{\end{lemma}}
\newcommand{\bd}{\begin{definition}}
\newcommand{\ed}{\end{definition}}
\newcommand{\bc}{\begin{corollary}}
\newcommand{\ec}{\end{corollary}}
\newcommand{\bp}{\begin{proof}}
\newcommand{\ep}{\end{proof}}
\newcommand{\bx}{\begin{example}}
\newcommand{\ex}{\end{example}}
\newcommand{\br}{\begin{remark}}
\newcommand{\er}{\end{remark}}
\newcommand{\be}{\begin{equation}}
\newcommand{\ee}{\end{equation}}
\newcommand{\ba}{\begin{align}}
\newcommand{\ea}{\end{align}}
\newcommand{\bn}{\begin{enumerate}}
\newcommand{\en}{\end{enumerate}}
\newcommand{\bcs}{\begin{cases}}
\newcommand{\ecs}{\end{cases}}
\newcommand{\RNum}[1]{\uppercase\expandafter{\romannumeral #1\relax}}

\makeatletter
\renewcommand{\section}{\@startsection{section}{1}{0mm}
  {-\baselineskip}{0.5\baselineskip}{\bf\leftline}}
\makeatother

\begin{document}
\sloppy
\title[IDEAL APPROXIMATION IN $n$-EXANGULATED CATEGORIES]{Ideal approximation in $n$-exangulated categories}
\author[Y. Wang, J. Wei]{Yucheng Wang, Jiaqun Wei}
\address{Institute of Mathematics, School of Mathematical Sciences, Nanjing Normal University,
 Nanjing 210023, P. R. China.\endgraf}
\email{wangyucheng2358@163.com; weijiaqun@njnu.edu.cn.}


\begin{abstract}
  In this paper, we study the ideal approximation theory associated to almost $n$-exact structures in the $n$-exangulated category. The notions of $n$-ideal cotorsion pairs and $n$-$\mathbb{F}$-phantom morphisms are introduced and studied. In particular, let $\mathscr{C}$ be an extriangulated category which satisfies the condition (WIC) and $\mathcal{T}$ be a nicely embedded $n$-cluster tilting subcategory of $\mathscr{C}$, we prove Salce's Lemma in $\mathcal{T}$.
\end{abstract}

\maketitle
\section{Introduction}
Recently, the study of higher homological algebra is an active topic, and its aim is to acquire a higher version of the classical homological theory. In order to build up a higher version of Auslander's correspondence and generalize the classical theory of almost split sequences, Iyama \cite{I} introduced the notion of $n$-cluster tilting subcategories for each positive integer $n$.

The study of $n$-cluster tilting subcategories in abelian categories and exact categories leads Jasso \cite{Ja} to define some new notions such as $n$-abelian and $n$-exact categories.
Similarly, Geiss, Keller and Oppermann \cite{GBS} introduced the notion of $(n+2)$-angulated categories with aim to study the $(n+2)$-cluster tilting subcategories in triangulated categories. Recently, Herschend, Liu and Nakaoka \cite{HLN1,HLN2} introduced the notion of $n$-exangulated categories for any positive integer $n$. It is not only a higher analogue of extriangulated categories defined by Nakaoka and Palu \cite{NP}, but also gives a reasonable generalization of $n$-exact categories and $(n+2)$-angulated categories. For the study of these higher categories, see for example \cite{HHZ1, HHZ2, HZ, ZW, ZZ}.

The approximation theory is one of the efficient tools to study the complicated objects by some simpler objects in a category. Approximation theory originates from the existence of injective envelopes by Baer in 1940. Due to the contributions of Auslander and his colleagues \cite{AR}, the approximation theory has played an important role in the representation theory of algebras. In the classical approximation theory, one used to consider the objects in some special subcategories. By a well-known embedding from a category to its morphism category, objects can be viewed as special morphisms. From the point of view, Fu, Guil Asensio, Herzog and Torrecillas \cite{FGHT} developed the ideal approximation theory. Furthermore, Fu and Herzog \cite{FH} studied the ideal versions of some results of the classical approximation theory.  The ideal approximation theory has been generalized to triangulated categories and extriangulated categories by Breaz--Modoi \cite{BM} and Zhao--Huang \cite{ZH}, respectively. Recently, the ideal approximation theory in the $n$-cluster tilting subcategories of exact categories was studied by Asadollahi and Sadeghi \cite{AS}; and the ideal approximation theory was also generalized to the extension closed subcategories of $n$-angulated categories in \cite{TWZ}. Salce's Lemma is one of the main theorems in the classical approximation theory. It relates the notions of (special) precoverings, (special) preenvelopings and cotorsion pairs. It was firstly introduced in the classical approximation theory in \cite{S}, and then its ideal versions in exact categories and triangulated categories were proved in \cite{FH} and \cite{BM}, respectively. The higher versions in some special $n$-exact categories and $n$-angulated categories also hold, see \cite{AS, TWZ}.

More generally, in this paper, we study the ideal approximation theory in $n$-exangulated categories. Note that the $n$-exact categories and $n$-angulated categories are the special cases of $n$-exangulated categories. Hence, our work generalizes some main results given in \cite{AS}, \cite{FGHT} and \cite{ZH}.

In Section 2, we give some terminologies and preliminary results, and recall the definitions of $n$-exangulated categories and nicely embedded $n$-cluster tilting subcategories. Moreover, we introduce the notion of almost $n$-exact structures in $n$-exangulated categories and give some examples of almost $n$-exact structures. The $n$-ideal cotorsion pairs and higher phantom morphisms in $n$-exangulated categories are defined and studied in Section 3. Section 4 is devoted to studing the connections between special precovering ideals and $n$-phantom morphisms in $n$-exangulated categories. In particular, we prove that under some conditions every special precovering ideal can be represented by {\bf Ph}($\mathbb{F}$) for some subfunctor $\mathbb{F}$. The higher version of Salce's Lemma (Theorem \ref{Thm3}) in nicely embedded $n$-cluster tilting subcategories of extriangulated categories is proved in Section 5. Finally, we give Theorem \ref{Thm5}, which can be viewed as the higher version of \cite[Theorem 1]{FGHT} in nicely embedded $n$-cluster tilting subcategories of extriangulated categories.

\section{Preliminaries}
In this section, we recall the $n$-exangulated categories and nicely embedded $n$-cluster tilting subcategories of extriangulated categories and introduce the almost $n$-exact structures in $n$-angulated categories.

\subsection{$n$-Exangulated categories}
We recall some definitions and basic properties of $n$-exangulated categories. For more details, the reader can refer to \cite{HLN1}. Throughout the subsection, let $\mathscr{C}$ be an additive category and $n$ be a positive integer. Suppose that $\mathscr{C}$ is equipped with an additive bifunctor
$\mathbb{E}:\mathscr{C}^{op} \times \mathscr{C} \rightarrow \mathsf{Ab}$. For any pair of objects $A,C\in \mathscr{C}$, an element $\delta\in \mathbb{E}(C,A)$ is called an $\mathbb{E}$-{\em extension} or simply an {\em extension}. We also write such $\delta$ as $_{A}\delta_{C}$ when we indicate $A$ and $C$. For any
pair of $\mathbb{E}$-{\em extensions} $_{A}\delta_{C}$ and $_{A'}{\delta'} _{C'}$, let $\delta \oplus \delta'\in \mathbb{E}(C \oplus C',A\oplus A')$ be the element corresponding
to $(\delta,0,0,\delta')$ through the natural isomorphism
$$
  \mathbb{E}(C \oplus C',A\oplus A') \simeq \mathbb{E}(C,A)\oplus \mathbb{E}(C',A) \oplus \mathbb{E}(C,A') \oplus \mathbb{E}(C',A').
$$

For any $a\in \mathscr{C}(A,A')$ and $c\in \mathscr{C}(C',C)$, we have $\mathbb{E}(C,a)(\delta)\in \mathbb{E}(C,A')$ and $\mathbb{E}(c,A)(\delta)\in \mathbb{E}(C',A)$, which are denoted by $a_{*}\delta$ and $c^{*}\delta$, respectively.

Let $_{A}\delta_{C}$ and $_{A'}{\delta'}_{C'}$ be any pair of $\mathbb{E}$-{\em extensions}. A {\em morphism} $(a,c):\delta \rightarrow \delta'$ of extensions is a pair of morphisms
$a\in \mathscr{C}(A,A')$ and $c\in \mathscr{C}(C,C')$ such that $a_* \delta =c^* \delta'$.
\begin{definition}
  {\rm \cite[Definition 2.7]{HLN1}}.\label{F1}
{\em Let $\mathbf{C}_\mathscr{C}$ be the category of complexes in $\mathscr{C}$. Define $\mathbf{C}^{n+2}_\mathscr{C}$ to be the full subcategory of $\mathbf{C}_\mathscr{C}$ consisting of the complexes whose components are zero in the degrees outside of $\{0,1,\cdots,n+1\}$. Namely, an object in $\mathbf{C}^{n+2}_\mathscr{C}$ is a complex $X^{\bullet}=\{X,d^{i}_{X}\}$ of the form
  $$
    X_{0}\stackrel{d^{0}_{X}}\longrightarrow X_{1}\stackrel{d^{1}_{X}}\longrightarrow X_{2}\stackrel{d^{2}_{X}}\longrightarrow\cdots \stackrel{d^{n-1}_{X}}\longrightarrow X_{n}\stackrel{d^{n}_{X}}\longrightarrow X_{n+1}.
  $$
We simply write a morphism $f^{\bullet}:X^{\bullet}\rightarrow Y^{\bullet}$ as $f^{\bullet}=(f^0,f^1,\cdots,f^{n+1})$, only indicating the terms of degrees $0,\dots,n+1$.}
\end{definition}

By Yoneda's lemma, any $\mathbb{E}$-extension $\delta\in \mathbb{E}(C,A)$ induces natural transformations
$$\delta_{\sharp}: \mathscr{C}(-,C)\Longrightarrow\mathbb{E}(-,A)~~\text{and}~~\delta^{\sharp}: \mathscr{C}(A,-)\Longrightarrow\mathbb{E}(C,-).$$ For any $X\in\mathscr{C}$, these
$(\delta_{\sharp})_X$ and $(\delta^{\sharp})_X$ are defined by
$(\delta_{\sharp})_X:  \mathscr{C}(X,C)\rightarrow\mathbb{E}(X,A), f\mapsto f^\ast\delta$ and $(\delta^{\sharp})_X:   \mathscr{C}(A,X)\rightarrow\mathbb{E}(C,X), g\mapsto g_\ast\delta$. In what follows, we may also simply write $\delta^{\sharp}_X(f)$ and $(\delta_{\sharp})_X(g)$ as $\delta^{\sharp}(f)$ and $\delta_{\sharp}(g)$, respectively.

\begin{definition}
  {\rm \cite[Definition 2.9]{HLN1}} \label{F2}
  {\em Let $\mathscr{C},\mathbb{E},n$ be as before. Define the category \AE$:=$\AE$_{(\mathscr{C},\mathbb{E})}^{(n+2)}$ as follows.

  {\rm (1)} A object in \AE$_{(\mathscr{C},\mathbb{E})}^{(n+2)}$ is a pair $\langle X^{\bullet},\delta\rangle $ consisting of $X^{\bullet}\in\mathbf{C}_\mathscr{C}^{n+2}$ and $\delta\in\mathbb{E}(X_{n+1},X_0)$ such that
  $$
    (d^0_X)_*\delta=0 \ \ and \ \ (d^n_X)^*\delta=0.
  $$
We call such a pair an $\mathbb{E}$-attached complex of length $n+2$. We also denote it by
  $$
    X_{0}\stackrel{d^{0}_{X}}\longrightarrow X_{1}\stackrel{d^{1}_{X}}\longrightarrow X_{2}\stackrel{d^{2}_{X}}\longrightarrow \cdots \stackrel{d^{n-1}_{X}}\longrightarrow  X_{n}\stackrel{d^{n}_{X}}\longrightarrow X_{n+1} \stackrel{\delta}\dashrightarrow.
  $$

  {\rm (2)} For any pairs $\langle X^{\bullet},\delta\rangle $ and $\langle Y^{\bullet},\rho\rangle $, a morphism $f^{\bullet}:\langle X^{\bullet},\delta\rangle \rightarrow \langle Y^{\bullet},\rho\rangle $ is defined to be a morphism $f^{\bullet}\in\mathbf{C}_\mathscr{C}^{n+2}(X^{\bullet},Y^{\bullet})$ satisfying $(f_0)_*\delta=(f_{n+1})^*\rho$.

  We use the same composition and the identities as in $\mathbf{C}^{n+2}_\mathscr{C}$.}
\end{definition}

\begin{definition}
  {\rm \cite[Definition 2.13]{HLN1}} \label{F3}
  {\em An $n$-exangle is a pair $\langle X^{\bullet},\delta\rangle $ consisting of $X^{\bullet}\in\mathbf{C}_\mathscr{C}^{n+2}$ and $\delta\in\mathbb{E}(X_{n+1},X_0)$ such that the following conditions hold.

  {\rm (1)} The following sequence of functors $\mathscr{C}^{op}\rightarrow \mathsf{Ab}$ is exact:
  $$
    \mathscr{C}(-,X_0)\stackrel{\mathscr{C}(-,d^{0}_{X})}\longrightarrow  \cdots \stackrel{\mathscr{C}(-,d^{n}_{X})}\longrightarrow \mathscr{C}(-,X_{n+1})\stackrel{\delta_{\sharp}}\longrightarrow \mathbb{E}(-,X_0).
  $$

  {\rm (2)} The following sequence of functors $\mathscr{C}\rightarrow \mathsf{Ab}$ is exact:
  $$
    \mathscr{C}(X_{n+1},-)\stackrel{\mathscr{C}(d^{n}_{X},-)}\longrightarrow  \cdots \stackrel{\mathscr{C}(d^{0}_{X},-)}\longrightarrow \mathscr{C}(X_0,-)\stackrel{\delta^{\sharp}}\longrightarrow \mathbb{E}(X_{n+1},-).
  $$
In particular, any $n$-exangle is an object in \AE. A $morphism$ of $n$-exangles simply means a
morphism in \AE. Thus $n$-exangles form a full subcategory of \AE.}
\end{definition}

\begin{lemma}
  {\rm \cite[Claim 2.15]{HLN1}} \label{lemma4}
  For any $n$-exangle $\langle X^{\bullet},\delta\rangle$, the following are equivalent.

  {\rm (1)} $\delta=0$.

  {\rm (2)} There is $r\in \mathscr{C}(X_1,A)$ such that $r\circ d_X^0=1_A$.

  {\rm (3)} There is $s\in \mathscr{C}(C,X_n)$ such that $d_X^n\circ s=1_C$.

\noindent If one of the above conditions holds, then we say that the $n$-exangle $\langle X^{\bullet},\delta\rangle$ splits.
\end{lemma}

\begin{definition}
  {\rm \cite[Definition 2.22]{HLN1}} \label{F4}
  {\em Let $\mathfrak{s}$ be a correspondence which associates a homotopic equivalence class $\mathfrak{s}(\delta)=[ _{A}{X^{\bullet}}_{C}]$ to each extension $\delta={_{A}\delta_{C}}$. Such $\mathfrak{s}$ is called a realization of $\mathbb{E}$ if it satisfies the following condition for any $\mathfrak{s}(\delta)=[X^{\bullet}]$ and $\mathfrak{s}(\rho)=[Y^{\bullet}]$.

  $\rm(R0)$ For any morphism of extensions $(a,c):\delta\rightarrow\rho$, there exists a morphism $f^{\bullet}\in\mathbf{C}_\mathscr{C}^{n+2}(X^{\bullet},Y^{\bullet})$ of the form $f^{\bullet}=(a,f_1,\cdots,f_n,c)$. Such $f^{\bullet}$ is called a lift of $(a,c)$.

  In such a case, we simply say that ``$X^{\bullet}$ realizes $\delta$'' whenever they satisfy $\mathfrak{s}(\delta)=[X^{\bullet}]$.

  Moreover, a realization $\mathfrak{s}$ of $\mathbb{E}$ is said to be $exact$ if it satisfies the following conditions.

  {\rm(R1)} For any $\mathfrak{s}(\delta)= [ X^\bullet ]$, the pair $\langle X^{\bullet},\delta\rangle $ is an $n$-exangle.

  {\rm(R2)} For any $A\in \mathscr{C}$, the zero element $_A {0}_0 =0 \in \mathbb{E}(0,A)$ satisfies
  $$
    \mathfrak{s}(_A{0}_0)=[ A \stackrel{1_A}\longrightarrow A \longrightarrow 0 \longrightarrow \cdots \longrightarrow 0 ].
  $$
Dually, $\mathfrak{s}(_0{0}_A)=[ 0 \longrightarrow 0 \longrightarrow \cdots \longrightarrow A \stackrel{1_A}\longrightarrow A ]$ holds for any $A\in \mathscr{C}$.}
\end{definition}

\begin{definition}
  {\rm \cite[Definition 2.23]{HLN1}} \label{F5}
  {\em Let $\mathbb{E}$ and  $\mathbf{C}_\mathscr{C}^{n+2}$ be as before, $\mathfrak{s}$ be an exact realization of $\mathbb{E}$. We use the following terminologies:

  {\rm (1)} An $n$-exangle $\langle X^{\bullet},\delta\rangle $ is called an $\mathfrak{s}$-{\em distinguished} $n$-exangle if it satisfies  $\mathfrak{s}(\delta)=[X^{\bullet}]$. We often simply say it is a distinguished $n$-exangle when $\mathfrak{s}$ is clear from the context.

  {\rm (2)} An object $X^{\bullet}\in\mathbf{C}_\mathscr{C}^{n+2}$ is called an $\mathfrak{s}$-{\em conflation} or simply a conflation if it realizes some extension $\delta\in\mathbb{E}(X_{n+1},X_0)$.

  {\rm (3)} A morphism $f$ in $\mathscr{C}$ is called an $\mathfrak{s}$-{\em inflation} or simply an inflation if it admits some conflation $X^{\bullet}\in\mathbf{C}_\mathscr{C}^{n+2}$ satisfying $d_0^X=f$.

  {\rm (4)} A morphism $g$ in $\mathscr{C}$ is called an $\mathfrak{s}$-{\em deflation} or simply a {\em deflation} if it admits some conflation $X^{\bullet}\in\mathbf{C}_\mathscr{C}^{n+2}$ satisfying $d_n^X=g$.}
\end{definition}

\begin{definition}
  {\rm \cite[Definition 2.27]{HLN1}} \label{F6}
  {\em For a morphism $f^{\bullet}\in\mathbf{C}_\mathscr{C}^{n+2}(X^{\bullet},Y^{\bullet})$ satisfying $f_0={\Id}_A$ for some $A=X_0=Y_0$, its {\em mapping cone} $M_f^{\bullet}\in\mathbf{C}_\mathscr{C}^{n+2}$ is defined to be the complex
  $$
    X_1 \stackrel{d^0_{M_f}}\rightarrow X_2\oplus Y_1 \stackrel{d^1_{M_f}}\rightarrow X_3\oplus Y_2 \stackrel{d^2_{M_f}}\rightarrow \cdots \stackrel{d^{n-1}_{M_f}}\rightarrow X_{n+1}\oplus Y_n \stackrel{d^{n}_{M_f}}\rightarrow Y_{n+1},
  $$
where $d^0_{M_f}=\begin{bmatrix} -d^1_X \\ f^1 \end{bmatrix}$, $d^i_{M_f}=\begin{bmatrix} -d^{i+1}_X & 0\\ f^{i+1} & d_Y^i \end{bmatrix} (1\leq i \leq n-1)$, $d^n_{M_f}=\begin{bmatrix} f^{n+1} & d_Y^n\end{bmatrix}$. The {\em mapping cocone} is defined dually, for morphisms $h^{\bullet}$ satisfying $h^{n+1}=\Id$ in $\mathbf{C}_\mathscr{C}^{n+2}$.}
\end{definition}

Now we can define the $n$-exangulated category.

\begin{definition}
  {\rm \cite[Definition 2.32]{HLN1}} \label{F7}
  {\em An $n$-exangulated category is a triplet $(\mathscr{C},\mathbb{E},\mathfrak{s})$ consisting of an additive category $\mathscr{C}$, biadditive functor $\mathbb{E}:\mathscr{C}^{op}\times \mathscr{C}\rightarrow \mathsf{Ab}$, and its exact realization $\mathfrak{s}$ such that the following conditions hold.\\
  {\rm (EA1) }Let $A\stackrel{f}\rightarrow B\stackrel{g}\rightarrow C$ be any sequence of morphisms in $\mathscr{C}$. If both $f$ and $g$ are inflations, then so is $g\circ f$. Dually, if $f$ and $g$ are deflations then so is $g\circ f$.\\
  {\rm(EA2)} For $\rho\in\mathbb{E}(D,A)$ and $c\in\mathscr{C}(C,D)$, let $_A{\langle X^{\bullet},c^*\rho\rangle}_C$ and $_A{\langle Y^{\bullet},\rho\rangle}_D$ be distinguished $n$-exangles. Then $(1_A,c)$ has a good lift $f^{\bullet}$, in the sense that its mapping cone gives a distinguished $n$-exangle $\langle M_f^{\bullet},(d^0_X)_*\rho\rangle$.\\
  {{\rm(EA}$2^{\rm op}$)} Dual of {\rm (EA2)}.}
\end{definition}

Here we recall some basic properties of the $n$-exangulated category $\mathscr{C}$.

\begin{lemma}
  {\rm \cite[Propsition 3.6]{HLN1}}\label{lemma1}
  Let $_A{\langle X^{\bullet},\delta\rangle}_C$ and $_B{\langle Y^{\bullet},\rho\rangle}_D$ be distinguished $n$-exangles. Suppose that we are given a commutative square
  $$
    \xymatrix{
     X_0 \ar[d]_-{a} \ar[r]^-{d^0_X} & X_1 \ar[d]^-{b}\\
     Y_0 \ar[r]^-{d^0_Y} & Y_1}
  $$
in $\mathscr{C}$. Then the following holds.

$(1)$ There is a morphism $f^{\bullet}:\langle X^{\bullet},\delta\rangle \rightarrow \langle Y^{\bullet},\rho\rangle$ satisfying $f^0=a$ and $f^1=b$.

$(2)$ If $X_0=Y_0=A$ and $a=1_A$ for some $A\in \mathscr{C}$, then the above $f^{\bullet}$ can be taken to give a distinguished $n$-exangle $\langle M_f^{\bullet},(d^0_X)_*\rho\rangle$.
\end{lemma}

\begin{lemma}\label{lemma2}
For any morphism of distinguished $n$-exangles:
    $$\xymatrix{
      A \ar[d]_-{a} \ar[r]^-{d^0_X} & X_1 \ar[d] \ar[r]^-{d^1_X} & X_2 \ar[d] \ar[r]^-{d^2_X} & \cdots  \ar[r]^-{d^{n-1}_X} & X_n \ar[d] \ar[r]^-{d^n_X} & C \ar[d]_-{c} \ar@{-->}^{\gamma}[r] & \\
      B \ar[r]^-{d^0_Y} & Y_1 \ar[r]^-{d^1_Y} & Y_2 \ar[r]^-{d^2_Y} & \cdots \ar[r]^-{d^{n-1}_Y} & Y_n \ar[r]^-{d^n_Y} & D \ar@{-->}^{\eta}[r] &
      }$$
we have that $a$ factors through $d^0_X$ if and only if $c$ factors through $d^n_Y$.
  \begin{proof}
  We only prove the necessity, since the sufficiency can be proved dually.
     Since $a$ factors through $d^0_X$, there exists $w:X_1\rightarrow B$ such that $wd^0_X=a$. Since every $n$-exangle belongs to \AE, we have that ${d_X^0}_*\gamma=0$. By the definition of $n$-exangle morphisms, we know that $a_*\gamma=c^*\eta$, so we have that
    $$
    a_*\gamma =(wd^0_X)_*\gamma=w_*(d^0_{X*}\gamma)= w_*(0)=0.
    $$
  Finally, consider the exact sequence $\mathscr{C}(C,Y_n) \stackrel{\mathscr{C}(C,d^{n}_{Y})}\longrightarrow \mathscr{C}(C,D)\stackrel{\eta_{\sharp}}\longrightarrow \mathbb{E}(C,B)$, there exists $v:C\rightarrow Y$ such that $d^n_Yv=c$, as we desired.
  \end{proof}
\end{lemma}

\subsection{$n$-Cluster tilting subcategories}

Following \cite{HLN2}, we recall the notion of $n$-cluster tilting subcategories in extriangulated categories. Throughout this subsection, let $(\mathscr{C},\mathbb{E},\mathfrak{s})$ be a $1$-exangulated category, or equivalently, an extriangulated category. Assume that it has enough projectives and injectives in the sense of \cite[Definition 1.13]{LN}, and denote by $\mathcal{P}$ and $\mathcal{I}$ the full subcategories of {projectives} and {injectives}, respectively. Define a biadditive functor $\mathbb{E}^i\colon \mathscr{C}^{op}\times \mathscr{C}\rightarrow \mathsf{Ab}$ to be the composition of
$$
  \mathscr{C}^{op}\times\mathscr{C}\rightarrow \underline{\mathscr{C}}^{op}\times \overline{\mathscr{C}} \stackrel{\Id\times\Sigma^{i-1}}\longrightarrow \underline{\mathscr{C}}^{op}\times \overline{\mathscr{C}} \stackrel{\mathbb{E}}\rightarrow \mathsf{Ab},
$$
where $\underline{\mathscr{C}}$ and $\overline{\mathscr{C}}$ are the stable categories $\mathscr{C}/\mathcal{P}$ and $\mathscr{C}/\mathcal{I}$, respectively; $\Sigma$ is the syzygy functor (see \cite[Assumption 3.3]{HLN2}). For any positive integers $i,j$ and any $A,C,X\in\mathscr{C}$, one defines the cup product $\cup\colon \mathbb{E}^i(X,A)\times \mathbb{E}^j(C,X)\rightarrow \mathbb{E}^{i+j}(C,A)$ by $\delta\cup\theta=\mathbb{E}^j(C,\overline{\delta})(\theta)$ for any pair $(\delta,\theta)\in\mathbb{E}^i(X,A)\times \mathbb{E}^j(C,X)$ and $\overline{\delta}\in\overline{\mathscr{C}}(X,\Sigma^i A)$. The specific details for these notions can be found in \cite[Section 3]{HLN2}.

\begin{definition}
  {\rm \cite[Definition 3.21]{HLN2}}\label{F8}
  {\em Let $\mathcal{T}\subseteq \mathscr{C}$ be a full additive subcategory closed under isomorphisms and direct summands. Such $\mathcal{T}$ is called an $n$-$cluster$ $tilting$ subcategory of $\mathscr{C}$, if it satisfies the following conditions.

  {\rm (1)} $\mathcal{T}$ is functorially finite.

  {\rm (2)} For any $C\in \mathscr{C}$, the following are equivalent.

  \ \ \ \ {\rm (i)} $C\in \mathcal{T}$.

  \ \ \ \ {\rm (ii)} $\mathbb{E}^i(C,\mathcal{T})=0$ for any $1\leq i \leq n-1$.

  \ \ \ \ {\rm (iii)} $\mathbb{E}^i(\mathcal{T},C)=0$ for any $1\leq i \leq n-1$.}
\end{definition}
Note that the $1$-cluster tilting subcategory is just the whole category $\mathscr{C}$.
So, we only consider the case for $n\geq 2$.
\begin{remark}
  {\rm\cite{HLN2}}\label{G1}
  Let $\mathcal{T}$ be an $n$-cluster tilting subcategory of an extriangulated category $\mathscr{C}$. Then we have that:

  {\rm (1)} For any $C\in \mathscr{C}$, there is a right $\mathcal{T}$-approximation $g_C\colon T_C\rightarrow C$ which is an $\mathfrak{s}$-deflation;

  {\rm (2)} Dually, any $A\in \mathscr{C}$ admits a left $\mathcal{T}$-approximation which is an $\mathfrak{s}$-inflation.
\end{remark}
\begin{definition}
  {\rm\cite[Definition 3.23]{HLN2}}\label{F9}
  {\em An $n$-cluster tilting subcategory $\mathcal{T}\subseteq \mathscr{C}$ is nicely embedded if the following conditions are satisfied.

  {\rm (1)} If $C\in \mathscr{C}$ satisfies $\mathbb{E}^{n-1}(\mathcal{T},C)=0$, then there is an $\mathfrak{s}$-triangle
  $$
    D\stackrel{q}\rightarrow P \rightarrow C \dashrightarrow
  $$ with $P\in \mathcal{P}$ such that
  $$\mathscr{C}(T,q)\colon \mathscr{C}(T,D)\rightarrow\mathscr{C}(T,P)$$
  is injective for any $T\in \mathcal{T}$.

  {\rm (2)} Dually, if $A\in \mathscr{C}$ satisfies $\mathbb{E}^{n-1}(A,\mathcal{T})=0$, then there is an $\mathfrak{s}$-triangle
  $$
    A\rightarrow I \stackrel{j}\rightarrow S \dashrightarrow
  $$ with $I\in \mathcal{I}$ such that
  $$\mathscr{C}(j,T)\colon \mathscr{C}(S,T)\rightarrow\mathscr{C}(I,T)$$
  is injective for any $T\in \mathcal{T}$.}
\end{definition}
We assume the following condition for the rest of the paper (see \cite[Condition 5.8]{NP}).

\noindent{\bf Condition} (WIC){\bf.}
  {\rm (1)} Let $f:X\rightarrow Y$ and  $g:Y\rightarrow Z$ be any pair of morphisms in $\mathscr{C}$. If $gf$ is an inflation, then $f$ is an inflation.

  {\rm (2)} Let $f:X\rightarrow Y$ and  $g:Y\rightarrow Z$ be any pair of morphisms in $\mathscr{C}$. If $gf$ is a deflation, then $g$ is a deflation.

\begin{definition}
  {\rm\cite[Definition 3.29]{HLN2}}\label{F10}
  {\em Let $\mathcal{T}\subseteq \mathscr{C}$ be an $n$-cluster tilting subcategory and $A,C\in\mathcal{T}$. Let $\delta\in\mathbb{E}^n(C,A)$. We say that an object $\langle X^{\bullet},\delta\rangle \in$\AE$_{(\mathcal{T},\mathbb{E}^n)}^{(n+2)}$
  \begin{equation}\label{fml1}
    A\stackrel{d^{0}_{X}}\longrightarrow X_{1}\stackrel{d^{1}_{X}}\longrightarrow X_{n-1}\stackrel{d^{2}_{X}}\longrightarrow \cdots \stackrel{d^{n-1}_{X}}\longrightarrow  X_{n}\stackrel{d^{n}_{X}}\longrightarrow C \stackrel{\delta}\dashrightarrow  \ \ \ \ (X_0=A,X_{n+1}=C,X_i\in \mathcal{T})
  \end{equation}
  is $\mathfrak{s}$-decomposable if it admits $\mathbb{E}$-triangles
  \begin{equation}\label{fml2}
    \begin{cases}
      A\stackrel{d^{0}_{X}}\longrightarrow X_{1}\stackrel{e^1}\longrightarrow M_1 \stackrel{\delta_{(1)}}\dashrightarrow,\\
      M_i\stackrel{m^{i}}\longrightarrow X_{i+1}\stackrel{e^{i+1}}\longrightarrow M_{i+1} \stackrel{\delta_{(i+1)}}\dashrightarrow \ \ (i=1,\cdots,n-2),   \\
      M_{n-1}\stackrel{m^{n-1}}\longrightarrow X_{n}\stackrel{d^n_X}\longrightarrow C \stackrel{\delta_{(n)}}\dashrightarrow
    \end{cases}
  \end{equation}
  satisfying $d^i_X=m^ie^i$ for $1\leq i\leq n-1$ and $\delta=\delta_{(1)}\cup\cdots\cup\delta_{(n)}$. We call (\ref{fml2}) an $\mathfrak{s}$-decomposition of (\ref{fml1}), as depicted specifically below.
    $$\xymatrix@=0.25cm{ &  &  &  &  &  &  &  &  & & & & & \\
    \ & \ & \ & \ &  & M_2\ar@{-->}^-{\delta_{(2)}}[ur]\ar[dr]^-{m^2} & & & & M_{n-1} \ar@{-->}^-{\delta_{(n-1)}}[ur]\ar[dr]^-{m^{n-1}} \\
   A \ar[rr]^-{d_X^0} & & X_1 \ar[rr]^-{d_X^1} \ar[dr]_-{e^1} &  & X_2 \ar[ur]^-{e^2} \ar[rr]^-{d_X^2} & & \cdots\cdots \ar[rr]^-{d_X^{n-2}} \ar[dr]_-{e^{n-2}} & & X_{n-2} \ar[ur]^-{e^{n-1}} \ar[rr]^-{d_X^{n-1}} & & X_n \ar[rr]^-{d_X^n} & & C \ar@{-->}^-{\delta_{(n)}}[rr] & & \ \\
   \ & \ & \ & M_1 \ar[ur]_-{m^1} \ar@{-->}_-{\delta_{(1)}}[dr] & & & & M_{n-2} \ar[ur]_-{m^{n-2}} \ar@{-->}_-{\delta_{(n-2)}}[dr] \\
    &  &  &  &  & & & & & & & &
   }$$}
\end{definition}

The examples of nicely embedded $n$-cluster tilting subcategories can be found in \cite[Example 3.24, Sections 4.2 and 4.3]{HLN2}. Assume that the extriangulated category $(\mathscr{C},\mathbb{E},\mathfrak{s})$ satisfies the Condition (WIC), and that its $n$-cluster tilting subcategory $\mathcal{T}\subseteq \mathscr{C}$ is nicely embedded. To endow $\mathcal{T}\subseteq \mathscr{C}$ with the structure of an $n$-exangulated category, for any $A,C\in\mathcal{T}$ and any $\delta\in\mathbb{E}^n(C,A)$, we define $\mathfrak{s}^n(\delta)=[ X^\bullet ]$ to be the homotopy equivalence class of $X^\bullet$ in $\mathscr{C}^{n+2}_{(\mathcal{T};A,C)}$, which gives an $\mathfrak{s}$-decomposable object $\langle X^{\bullet},\delta\rangle \in$\AE$_{(\mathcal{T},\mathbb{E}^n)}^{(n+2)}$. Then by \cite[Theorem 3.41]{HLN2}, $(\mathcal{T},\mathbb{E}^n,\mathfrak{s}^n)$ becomes an $n$-exangulated category.

\subsection{Almost $n$-exact structures}\label{AnS}

Through this subsection, let $(\mathscr{C}, \mathbb{E},\mathfrak{s})$ be an $n$-exangulated category. Assume that $\mathcal{T}$ is a full subcategory of $\mathscr{C}$ which is closed under extensions, i.e., for any $A,C\in \mathcal{T}$ and any $\delta\in \mathbb{E}(C,A)$, there exists an $\mathfrak{s}$-distinguished $n$-exangle $\langle X^{\bullet},\delta\rangle $ which satisfies $X_i\in\mathcal{T}$ for $1\leq i \leq n$. A class of distinguished $n$-exangles $\mathcal{F}$ in $\mathcal{T}$ is called an {\em almost} $n$-{\em exact structure} for $\mathcal{T}$ if it satisfies the following conditions:

{\rm (NE1)} $\mathcal{F}$ is closed under direct sums and contains all split $n$-exangles.

{\rm (NE2)} For any distinguished $n$-exangle $X_{0}\stackrel{d^{0}_{X}}\longrightarrow X_{1}\stackrel{d^{1}_{X}}\longrightarrow X_{2}\stackrel{d^{2}_{X}}\longrightarrow \cdots \stackrel{d^{n-1}_{X}}\longrightarrow  X_{n}\stackrel{d^{n}_{X}}\longrightarrow X_{n+1} \stackrel{\delta}\dashrightarrow$ in $\mathcal{F}$ and any morphism $f\colon Y_{n+1}\rightarrow X_{n+1}$, $f^*\delta\in\mathcal{F}$, i.e., we have the following commutative diagram
$$\xymatrix{
    X_0 \ar@{=}[d] \ar[r] & Y_1 \ar[d] \ar[r] & Y_2 \ar[d] \ar[r] & \cdots  \ar[r] & Y_n \ar[d] \ar[r] & Y_{n+1} \ar[d]_-{f} \ar@{-->}^{f^*\delta\in\mathcal{F}}[r] & \\
    X_0 \ar[r] & X_1 \ar[r] & X_2 \ar[r] & \cdots \ar[r] & X_n \ar[r] & X_{n+1} \ar@{-->}^{\delta\in\mathcal{F}}[r] & }$$

{\rm (NE3)} For any distinguished $n$-exangle $X_{0}\stackrel{d^{0}_{X}}\longrightarrow X_{1}\stackrel{d^{1}_{X}}\longrightarrow X_{2}\stackrel{d^{2}_{X}}\longrightarrow\cdots \stackrel{d^{n-1}_{X}}\longrightarrow  X_{n}\stackrel{d^{n}_{X}}\longrightarrow X_{n+1} \stackrel{\delta}\dashrightarrow$ in $\mathcal{F}$ and any $g\colon X_0\rightarrow Y_0$, then $g_*\delta\in\mathcal{F}$, i.e., we have the following commutative diagram
$$\xymatrix{
    X_0 \ar[d]^-{g} \ar[r] & X_1 \ar[d] \ar[r] & X_2 \ar[d] \ar[r] & \cdots  \ar[r] & X_n \ar[d] \ar[r] & X_{n+1} \ar@{=}[d] \ar@{-->}^{\delta\in\mathcal{F}}[r] & \\
    Y_0 \ar[r] & Y_1 \ar[r] & Y_2 \ar[r] & \cdots \ar[r] & Y_n \ar[r] & X_{n+1} \ar@{-->}^{g_*\delta\in\mathcal{F}}[r] & }
$$
By the definitions of additive subfunctor \cite[Subsection 3.2]{HLN1} and almost $n$-exact structure, we know that each almost $n$-exact structure $\mathcal{F}$ in $\mathscr{C}$ gives rise to an additive subfunctor $\mathbb{F}$ of $\mathbb{E}$. Conversely, any additive subfunctor $\mathbb{F}$ of $\mathbb{E}$ induces an almost $n$-exact structure. In the following, we give some examples about almost $n$-exact structures.
\begin{example}

  {\rm (i)} Let $(\mathscr{C}, \mathbb{E},\mathfrak{s})$ be an $n$-exangulated category and $\mathbb{F}$ be an additive subfunctor of $\mathbb{E}$. If $\mathbb{F}$ is closed in the sense of \cite[Definition 3.10]{HLN1}, then by \cite[Proposition 3.16]{HLN1}, $(\mathscr{C}, \mathbb{F},\mathfrak{s}|_{\mathbb{F}})$ is an $n$-exangulated category. If we take $\mathcal{F}$ to be the class of all distinguished $n$-$\mathbb{F}$-exangles, then $\mathcal{F}$ becomes an almost $n$-exact structure in $\mathscr{C}$.


  {\rm (ii)} Suppose that the additive category $\mathscr{C}$ is equipped with an automorphism $\Sigma\colon \mathscr{C}\stackrel{\cong}\rightarrow\mathscr{C}$. Then $\Sigma$ induces an additive bifunctor $\mathbb{E}_{\Sigma}=\mathscr{C}(-,\Sigma(-))\colon \mathscr{C}^{op}\times \mathscr{C}\rightarrow \mathsf{Ab}$. By \cite[Proposition 4.8]{HLN1}, define $\diamondsuit_{\mathfrak{s}}$ to be the class of $(n+2)$-$\Sigma$-sequences from distinguished $n$-exangles, then $(\mathscr{C},\mathbb{E}_{\Sigma},\diamondsuit_{\mathfrak{s}})$ becomes an $(n+2)$-angulated category. Consequently, the almost $n$-exact structure in $\mathscr{C}$ is the same as they defined in \cite{TWZ}.

  {\rm (iii)} Suppose that $(\mathscr{C}, \mathbb{E},\mathfrak{s})$ is an $n$-exangulated category in which any $\mathfrak{s}$-inflation is monomorphic and any $\mathfrak{s}$-deflation are epimorphic. Let $\mathcal{X}$ be the class of all $\mathfrak{s}$-conflations, then $(\mathscr{C},\mathcal{X})$ becomes an $n$-exact
  category (cf. \cite[Proposition 4.37]{HLN1}). In this case,  the almost $n$-exact structure in $\mathscr{C}$ is exactly the $n$-proper class in \cite[2.12]{AS}. For exact categories and abelian categories, see \cite[1.2]{DRSS}  and \cite{ASo}, respectively.
\end{example}

\section{Ideal Cotorsion Pairs and Higher Phantom Morphisms}
Let $(\mathscr{C}, \mathbb{E},\mathfrak{s})$ be an $n$-exangulated category. In this section, we introduce and study the ideal cotorsion pairs and higher phantom morphisms in $\mathscr{C}$.

\subsection{Precover and preenvelope ideals}
A {\em two sided ideal} $\mathfrak{I}$ of $\mathscr{C}$ is a subfunctor
$$
  \mathfrak{I}(-,-)\colon \mathscr{C}^{op}\times \mathscr{C} \rightarrow \mathbf{Ab}
$$
of the bifunctor $\mathscr{C}(-,-)$ that associates to every pair $(A,A')$, where $A,A'\in \mathscr{C}$, a subgroup $\mathfrak{I}(-,-)\subseteq \mathscr{C}(A,A')$ such that for any $f\colon X\rightarrow A, g\colon B\rightarrow Y$, $\mathscr{C}(f,g)(i)=gif\in\mathfrak{I}(X,Y)$ for any $i\in \mathfrak{I}(A,B)$. We call an object $A$ of $\mathscr{C}$ is in $\mathfrak{I}$ if the identity morphism $1_A\in \mathfrak{I}(A,A)$ and we define ${\bf Ob} \mathfrak{I}=\{ A\in\mathscr{C}~|~1_A\in\mathfrak{I} \} $.

Let $\mathfrak{I}$ be an ideal of $\mathscr{C}$ and $A$ be an object of $\mathscr{C}$. An $\mathfrak{I}$-{\em precover} of $A$ is a morphism $\phi\colon C\rightarrow A$ in $\mathfrak{I}$ such that for any $g\colon C'\rightarrow A$, $g$ factors through $\phi$, i.e., there exists $h\colon C'\rightarrow C$ such that $g=\phi h$.
The ideal $\mathfrak{I}$ is a {\em precovering ideal} if for any $A\in \mathscr{C}$, $A$ has an $\mathfrak{I}$-precover. Dually, we can define the notions of $\mathfrak{J}$-{\em preenvelope} and {\em preenveloping ideal}. Let $\mathcal{M}$ be a collection of morphisms in $\mathscr{C}$, we give the notions of left and right orthogonal ideals by setting
$$
  \mathcal{M}^{\perp }:= \{ g~|~\mathbb{E}(m,g)=0, \forall m\in \mathcal{M} \}
$$
and
$$
    ^{\perp }{\mathcal{M}}:= \{ f~|~\mathbb{E}(f,m)=0, \forall m\in \mathcal{M} \},
$$
respectively.

\begin{proposition}\label{Prop2}
  Let $\mathcal{M}$ be a collection of morphisms in $\mathscr{C}$, then both $\mathcal{M}^{\perp }$ and $^{\perp }{\mathcal{M}}$ are ideals of $\mathscr{C}$.
  \begin{proof}
    We just verify that $\mathcal{M}^{\perp}$ is an ideal. Take any distinguished $n$-exangle $X_{0}\longrightarrow X_{1}\longrightarrow X_{2}\longrightarrow \cdots \longrightarrow  X_{n}\longrightarrow X_{n+1} \stackrel{\delta}\dashrightarrow$. By the additvity of bifunctor $\mathbb{E}$, for any $g_1,g_2\colon X_0\rightarrow Y$ in $\mathcal{M}^{\perp}$ and $m\colon Y'\rightarrow X_{n+1}$, we have that \begin{flalign*}
      \mathbb{E}(m,g_1+g_2)(\delta)&=\mathbb{E}(m,Y)\circ\mathbb{E}(X_{n+1},g_1+g_2)(\delta)\\
      &=\mathbb{E}(m,Y)\circ(\mathbb{E}(X_{n+1},g_1)+\mathbb{E}(X_{n+1},g_2))(\delta)\\
      &=\mathbb{E}(m,g_1)(\delta)+\mathbb{E}(m,g_2)(\delta)=0,
    \end{flalign*} 
    which implies that $g_1+g_2\in \mathcal{M}^{\perp}$.

    It remains to show that $\mathbb{E}(m,g_1k)=0$ and $\mathbb{E}(m,hg_1)=0$ for any morphisms $h\colon Y\rightarrow V$ and $k\colon W\rightarrow X_0$. For any distinguished $n$-exangle $W\longrightarrow W_{1}\longrightarrow W_{2}\longrightarrow \cdots \longrightarrow  W_{n}\longrightarrow X_{n+1} \stackrel{\eta}\dashrightarrow$, we have that $\mathbb{E}(m,g_1k)(\eta)=\mathbb{E}(m,X_0)\circ\mathbb{E}(Y',g_1)\circ\mathbb{E}(Y',k)(\eta)=\mathbb{E}(m,g_1)\circ\mathbb{E}(Y',k)(\eta)=0$, which implies that $g_1k\in\mathcal{M}^{\perp}$. Similarly, we can show that $hg_1\in \mathcal{M}^{\perp}$. Then $\mathcal{M}^{\perp}$ is an ideal.
  \end{proof}
\end{proposition}


\begin{definition}\label{F11}
  {\em Let $\mathfrak{I}$ and $\mathfrak{J}$ be two ideals of $\mathscr{C}$. A pair $(\mathfrak{I},\mathfrak{J})$ is called an $n$-orthogonal pair of ideals if for any $f\in\mathfrak{I}$ and $g\in\mathfrak{J}$, $\mathbb{E}(f,g)=0$. A pair $(f,g)$ which satisfies the above condition is called an $n$-orthogonal pair of morphisms.}
\end{definition}

\begin{remark}
  {\rm (1)} Taken an ideal $\mathfrak{I}$ in $\mathscr{C}$, it is easy to construct two $n$-orthogonal pairs $(\mathfrak{I},\mathfrak{I}^{\perp})$ and $(\mathfrak{I},{^{\perp }\mathfrak{I}})$.

  {\rm (2)} The pair $(1_{X_0},g)$ is an $n$-orthogonal pair of morphisms if and only if $\mathbb{E}(X_0,g)=0$; Dually, the pair $(f,1_{X_{n+1}})$ is an $n$-orthogonal pair of morphisms if and only if $\mathbb{E}(f,X_{n+1})=0$; The pair $(1_{X_0},1_{X_{n+1}})$ is an $n$-orthogonal pair of morphisms if and only if $\mathbb{E}(X_0,X_{n+1})=0$.
\end{remark}

\begin{definition}
  {\em The $n$-orthogonal pair $(\mathfrak{I},\mathfrak{J})$ of ideals in $\mathscr{C}$ is called an $n$-ideal cotorsion pair if $\mathfrak{I}={^{\perp}\mathfrak{J}}$ and $\mathfrak{J}=\mathfrak{I}^{\perp}$.}
\end{definition}

Note that given ideals $\mathfrak{I}$ and $\mathfrak{J}$ in $\mathscr{C}$, in general, the $n$-orthogonal pairs $(\mathfrak{I},\mathfrak{I}^{\perp})$ and $(\mathfrak{I},{^{\perp }\mathfrak{I}})$ are not $n$-ideal cotorsion pairs unless $\mathfrak{I}={^{\perp}(\mathfrak{I}^{\perp})}$ and $\mathfrak{J}={^{\perp}(\mathfrak{J}^{\perp})}$, respectively. It is natural to hope that $(\mathfrak{I},\mathfrak{I}^{\perp})$ becomes an $n$-ideal cotorsion pair. For this purpose, we introduce the notion of special precovering ideals. This is an analogue of {\rm\cite[Section 1]{FGHT}}.

\begin{definition}\label{F12}
  {\em Let $\mathfrak{I}$ be an ideal of $\mathscr{C}$ and $C\in\mathscr{C}$ be any object. A morphism $i\colon Y_n\rightarrow C$ in $\mathfrak{I}$ is called a special $\mathfrak{I}$-precover of $C$ if it is the $\mathfrak{s}$-deflation of distinguished $n$-exangle which realizing $\mathbb{E}(C,j)(\delta)$ for some $\delta\in\mathbb{E}(C,A)$ and $j\colon A\rightarrow A'$ in $\mathfrak{I}^{\perp}$, i.e., we have the following commutative diagram
  $$
    \xymatrix{
      A \ar[d]_-{j\in\mathfrak{I}^{\perp}} \ar[r] & X_1 \ar[d] \ar[r] & X_2 \ar[d] \ar[r] & \cdots  \ar[r] & X_n \ar[d] \ar[r] & C \ar@{=}[d] \ar@{-->}^{\delta}[r] & \\
      A' \ar[r] & Y_1 \ar[r] & Y_2 \ar[r] & \cdots \ar[r] & Y_n \ar[r]^{i} & C \ar@{-->}^{j_*\delta}[r] & . }
  $$
  The ideal $\mathfrak{I}$ is called a special precovering ideal if each object $C\in\mathscr{C}$ has a special $\mathfrak{I}$-precover. Dually, we can define the special $\mathfrak{J}$-preenvelopes and special preenveloping ideals.}
\end{definition}

\begin{definition}\label{F13}
  {\em An $n$-ideal cotorsion pair $(\mathfrak{I},\mathfrak{J})$ is called complete if for any $C\in \mathscr{C}$, $C$ has a special $\mathfrak{I}$-precover and a special $\mathfrak{J}$-preenvelope.}
\end{definition}

\begin{proposition}\label{Prop3}
  If $i$ is a special $\mathfrak{I}$-precover (resp. special $\mathfrak{J}$-preenvelope), then $i$ is a $\mathfrak{I}$-precover (resp. $\mathfrak{J}$-preenvelope).
\end{proposition}
\begin{proof}
We only prove the case of special $\mathfrak{I}$-precover. The proof of special $\mathfrak{J}$-preenvelope is dual. Since $i$ is a special $\mathfrak{I}$-precover, there exists $\delta\in\mathbb{E}(C,A)$ and $j\in\mathfrak{I}^{\perp}$ such that we have the upper half part of the following commutative diagram:
      $$\xymatrix{
        A \ar[d]_-{j} \ar[r] & X_1 \ar[d] \ar[r] & X_2 \ar[d] \ar[r] & \cdots  \ar[r] & X_n \ar[d] \ar[r] & C \ar@{=}[d] \ar@{-->}^{\delta}[r] & \\
        A' \ar[r] & Y_1 \ar[r] & Y_2 \ar[r] & \cdots \ar[r] & Y_n \ar[r]^{i} & C \ar@{-->}^{j_*\delta}[r] & \\
        A' \ar@{=}[u] \ar[r] & Z_1 \ar[u]^-{f^1} \ar[r] & Z_2 \ar[u]^-{f^2} \ar[r] & \cdots \ar[r] & Z_n  \ar[u]^-{f^n} \ar[r]^-{d_Z^n} & Y \ar[u]^-{f} \ar@{-->}^-{f^*j_*\delta}[r] & . }$$
    For any morphism $f\colon Y\rightarrow C$ in $\mathfrak{I}$, $\mathbb{E}(f,j)=0$, thus $f^*j_*\delta=0$. So the third row in the above diagram splits. By Lemma \ref{lemma4}, there exists $r\colon Y\rightarrow Z_n$ such that $d_Z^n r=\Id$. It follows that $if^nr=f$,
   which proves that $i$ is a $\mathfrak{I}$-precover.
\end{proof}


\begin{theorem}\label{Thm2}
  Let $\mathfrak{I}$ be a special $\mathfrak{I}$-precover of $\mathscr{C}$, then the $n$-orthogonal pair of ideals $(\mathfrak{I},\mathfrak{I}^{\perp})$ is an $n$-ideal cotorsion pair.
  \begin{proof}
    We just need to show that ${^{\perp}(\mathfrak{I}^{\perp})}\subseteq\mathfrak{I}$. Take any morphism $i'\colon C'\rightarrow C$ in ${^{\perp}(\mathfrak{I}^{\perp})}$, since $\mathfrak{I}$ is special precovering,  there exists $\delta\in\mathbb{E}(C,A)$ and $j\in\mathfrak{I}^{\perp}$ such that we have the following commutative diagram
      $$\xymatrix{
        A \ar[d]_-{j} \ar[r] & X_1 \ar[d] \ar[r] & X_2 \ar[d] \ar[r] & \cdots  \ar[r] & X_n \ar[d] \ar[r] & C \ar@{=}[d] \ar@{-->}^{\delta}[r] & \\
        A' \ar[r] & Y_1 \ar[r] & Y_2 \ar[r] & \cdots \ar[r] & Y_n \ar[r]^-{i} & C \ar@{-->}^{j_*\delta}[r] & ,
      }$$
    where $i\colon Y_n\rightarrow C$ is an $\mathfrak{I}$-precover by Proposition \ref{Prop3}. Note that $\mathbb{E}(i',j)=0$, since $i'\in{^{\perp}(\mathfrak{I}^{\perp})}$. Thus, we get that
    $$
      \mathbb{E}(i',A')(j_*\delta)=\mathbb{E}(i',A')\circ\mathbb{E}(C,j)(\delta)=\mathbb{E}(i',j)(\delta)=0.
    $$
    Hence, there exists $h\colon C'\rightarrow Y_n$ such that $i'=ih$.
    Therefore, $i'\in \mathfrak{I}$.
  \end{proof}
\end{theorem}

\subsection{Higher phantom morphisms in $n$-exangulated categories}
In the subsection, we introduce and study the basic properties of higher phantom morphisms. 

Throughout the subsection, $\mathbb{F}$ denotes an additive subfunctor of $\mathbb{E}$. By the discussion in subsection \ref{AnS}, we have a one-to-one correspondence between the additive subfunctors and almost $n$-exact structures. Hence we also use $\mathbb{F}$ to denote an almost $n$-exact structure. For any $\mathfrak{s}(\delta)=[ A\stackrel{d^{0}_{X}}\longrightarrow X_{1}\stackrel{d^{1}_{X}}\longrightarrow \cdots \stackrel{d^{n-2}_{X}}\longrightarrow X_{n-1}\stackrel{d^{n-1}_{X}}\longrightarrow X_{n}\stackrel{d^{n}_{X}}\longrightarrow C]$ with $\delta\in \mathbb{F}(C,A)$, $d^0_X$ and $d^n_X$ are called $\mathfrak{s}|_\mathbb{F}$-{\em inflation} and $\mathfrak{s}|_\mathbb{F}$-{\em deflation}, respectively. For the simplicity of notation, we write $\mathfrak{s}|_\mathbb{F}$ as $\mathbb{F}$ and call them $\mathbb{F}$-{\em inflations} and $\mathbb{F}$-{\em deflations}, respectively.

\begin{definition}\label{F14}
  {\em Given an additive subfunctor $\mathbb{F}$ of $\mathbb{E}$, a morphism $\phi\colon X\rightarrow C $ is called an $n$-$\mathbb{F}$-phantom-morphism if for any $\delta\in \mathbb{E}(C,A)$, we have that $\phi^*\delta\in \mathbb{F}(X,A)$. It can be depicted by the following commutative diagram
    $$\xymatrix{
      A \ar@{=}[d] \ar[r] & Y_1 \ar[d] \ar[r] & Y_2 \ar[d] \ar[r] & \cdots  \ar[r] & Y_n \ar[d] \ar[r] & X \ar[d]^-{\phi} \ar@{-->}^{\phi^*\delta\in\mathbb{F}}[r] & \\
      A \ar[r] & X_1 \ar[r] & X_2 \ar[r] & \cdots \ar[r] & X_n \ar[r] & C \ar@{-->}^-{\delta}[r] & .
    }$$
  The collection of all $n$-$\mathbb{F}$-phantom-morphisms is denoted to ${\bf Ph}(\mathbb{F})$. The $n$-$\mathbb{F}$-cophantom-morphisms and ${\bf CoPh}(\mathbb{F})$ are defined dually.}
\end{definition}

We can check that ${\bf Ph}(\mathbb{F})$ and ${\bf CoPh}(\mathbb{F})$ are ideals of $\mathscr{C}$: (i) For any morphism $\phi\colon X\rightarrow C$ in ${\bf Ph}(\mathbb{F})$, $\delta\in \mathbb{E}(W,A)$ and $f\in \mathscr{C}(C,W)$, we have that $\phi^*f^*\delta=(f\phi)^*\delta\in\mathbb{F}(X,A)$. Hence, $f\phi\in {\bf Ph}(\mathbb{F})$; (ii) for any $\rho\in \mathbb{E}(C,A)$ and $g\in \mathscr{C}(V,X)$, we have that $(\phi g)^*\rho=g^*\phi^*\rho$. Since $\phi\in {\bf Ph}(\mathbb{F})$, $\phi^*\rho\in \mathbb{F}(X,A)$, we get that $(\phi g)^*\rho\in \mathbb{F}(V,A)$; (iii) By the additivity of $\mathbb{F}$, it is easy to see that $\phi+\psi \in {\bf Ph}(\mathbb{F})$, if $\phi,\psi\in\mathscr{C}(X,C)\cap {\bf Ph}(\mathbb{F})$.

Given an ideal $\mathfrak{I}$, define {\bf PB}$(\mathfrak{I})$ to be a collection of distinguished $n$-exangles that realizes the $\mathbb{E}(f,A)(\delta)$ for some $\delta\in\mathbb{E}(C,A)$ with $C,A\in\mathscr{C}$ and morphism $f$ in $\mathfrak{I}$.

\begin{proposition}\label{Prop4}
For any ideal $\mathfrak{I}$ of $\mathcal{T}$, ${\bf PB}(\mathfrak{I})$ is an almost $n$-exact structure in $\mathscr{C}$.
  \begin{proof}
    By the definition of almost $n$-exact structures, we need to verify ${\bf PB}(\mathfrak{I})$ contains all split $n$-exangles, is closed under isomorphisms and direct sums, and satisfies {\rm (NE2)} and {\rm (NE3)}.

    (i) Consider two isomorphic distinguished $n$-exangles which realize $\delta\in\mathbb{E}(X,A)$ and $\rho\in\mathbb{E}(Y,B)$, respectively. If the distinguished $n$-exangle realizing $\delta$ belongs to ${\bf PB}(\mathfrak{I})$, then there exists $\eta\in\mathbb{E}(C,A)$ and $f\in\mathscr{C}(X,C)$ such that $f^*\eta=\delta$. Thus, $(fg)^*(a_*\eta)=\rho$ and the distinguished $n$-exangle realizing $\rho$ belongs to ${\bf PB}(\mathfrak{I})$.
      $$\xymatrix{
        B \ar[d]_-{\wr} \ar[r] & Z_1 \ar[d]_-{\wr} \ar[r] & Z_2 \ar[d]_-{\wr} \ar[r] & \cdots \ar[r] & Z_n \ar[d]_-{\wr} \ar[r] & Y \ar[d]_-{\wr}^-{g} \ar@{-->}^{\rho}[r] & \\
        A \ar@{=}[d] \ar[r] & Y_1 \ar[d] \ar[r] & Y_2 \ar[d] \ar[r] & \cdots \ar[r] & Y_n \ar[d] \ar[r] & X \ar[d]_-{f} \ar@{-->}^{\delta}[r] & \\
        A \ar[d]_-{\wr}^-{a} \ar[r] & X_1 \ar[d] \ar[r] & X_2 \ar[d] \ar[r] & \cdots \ar[r] & X_n \ar[d] \ar[r] & C \ar@{=}[d] \ar[r] \ar@{-->}^{\eta}[r] & \\
        B \ar[r] & U_1 \ar[r] & U_2 \ar[r] & \cdots \ar[r] & U_n \ar[r] & C \ar@{-->}^-{a_*\eta}[r] & .
      }$$

    (ii) For a split $n$-exangle
    $$
      A\longrightarrow X_{1}\longrightarrow X_2 \longrightarrow \cdots \longrightarrow X_{n}\longrightarrow C \stackrel{\delta}\dashrightarrow,
    $$
    we take the zero morphism $0\colon C\rightarrow C$ in $\mathfrak{I}$, then $0^*(\delta)=\delta$. Hence, ${\bf PB}(\mathfrak{I})$ contains all split $n$-exangles.

    (iii) For any two distinguished $n$-exangles in ${\bf PB}(\mathfrak{I})$ which realize $\delta$ and $\delta'$, respectively, there exists $\mathfrak{s}$-extensions $\eta$, $\eta'$ and morphisms $i\colon C\rightarrow B$, $i'\colon C'\rightarrow B'$ in $\mathfrak{I}$ such that $\mathbb{E}(i,A)(\eta)=\delta$ and $\mathbb{E}(i',A')(\eta')=\delta'$, respectively. Then the direct sum $\delta\oplus\delta'$ of $\delta$ and $\delta'$ is denoted by $\mathbb{E}(\begin{bmatrix} i & 0\\ 0 & i' \end{bmatrix}, A\oplus A')(\eta\oplus\eta')$. This shows that ${\bf PB}(\mathfrak{I})$ is closed under direct sums.

    (iv) For any $\mathfrak{s}$-conflation $A\longrightarrow X_{1}\longrightarrow X_{2}\longrightarrow\cdots \longrightarrow  X_{n}\longrightarrow C \stackrel{\delta}\dashrightarrow$ in ${\bf PB}(\mathfrak{I})$, there exists $A\longrightarrow Y_{1}\longrightarrow Y_2 \longrightarrow \cdots \longrightarrow Y_{n}\longrightarrow X \stackrel{\eta}\dashrightarrow$ and $i\colon C\rightarrow X\in\mathfrak{I}$ such that $\mathbb{E}(i,A)(\eta)=\delta$. For any $f\colon W\rightarrow C$, we have $\mathbb{E}(f,A)\circ\mathbb{E}(i,A)(\eta)=\mathbb{E}(if,A)(\eta)$. This shows {\rm (NE2)}.

    (v) We use the same assumption in (iv), i.e. $\mathbb{E}(i,A)(\eta)=\delta$. For any $g\colon A\rightarrow V$, we have
    $$
      \mathbb{E}(C,g)(\delta)=\mathbb{E}(X,g)\circ\mathbb{E}(i,A)(\eta)=\mathbb{E}(i,V)(\mathbb{E}(X,g)(\eta)).
    $$
    This shows {\rm (NE3)}.
  \end{proof}
\end{proposition}
For the almost $n$-exact structure ${\bf PB}(\mathfrak{I})$ in Proposition \ref{Prop4}, we have the corresponding subfunctor, which is also denoted by ${\bf PB}(\mathfrak{I})$.

\begin{definition}\label{F15}
  {\em A morphism $f\colon X\rightarrow C$ in $\mathscr{C}$ is called $\mathbb{F}$-projective if for every object $A$ in $\mathscr{C}$, $\mathbb{F}(f,A)=0$, i.e., we have the following morphism of distinguished $n$-exangles
    $$\xymatrix{
      A \ar@{=}[d] \ar[r] & Y_1 \ar[d] \ar[r] & Y_2 \ar[d] \ar[r] & \cdots  \ar[r] & Y_n \ar[d] \ar[r] & X \ar[d]^-{f} \ar@{-->}^{f^*\delta=0}[r] & \\
      A \ar[r] & X_1 \ar[r] & X_2 \ar[r] & \cdots \ar[r] & X_n \ar[r] & C \ar@{-->}^-{\delta}[r] & .
    }$$
An object $C$ in $\mathscr{C}$ is called $\mathbb{F}$-projective if $1_C$ is an $\mathbb{F}$-projective morphism. It is easy to see that the collection of all $\mathbb{F}$-projective morphisms is an ideal, denoted by $\mathbb{F}$-{\bf proj}. Similarly, the notions of $\mathbb{F}$-injective morphisms and $\mathbb{F}$-injective objects are defined dually. The ideal of $\mathbb{F}$-injective morphisms is denoted by $\mathbb{F}$-{\bf inj}.}
\end{definition}

Let $\mathfrak{I}$ be an ideal of $\mathcal{T}$ and set $\mathbb{F}={\bf PB}(\mathfrak{I})$. Noting that $j\in\mathfrak{I}^{\perp}\Leftrightarrow \mathbb{E}(i,j)=0$ for any $i\in\mathfrak{I} \Leftrightarrow j\in {\bf PB}(\mathfrak{I})$-${\bf inj}$, we have that $\mathbb{F}$-${\bf inj}=\mathfrak{I}^{\perp}$.

\begin{proposition}\label{Prop5}
  For any special precovering ideal $\mathfrak{I}$, $\mathfrak{I}$ equals to the ideal ${\bf Ph}({\bf PB}(\mathfrak{I}))$ of $n$-${\bf PB}(\mathfrak{I})$-phantom morphisms.
  \begin{proof}
    By the definition of ${\bf PB}(\mathfrak{I})$, it is easy to see that each $i\in\mathfrak{I}$ is an $n$-${\bf PB}(\mathfrak{I})$-phantom morphism. It remains to prove each $n$-${\bf PB}(\mathfrak{I})$-phantom morphism $i$ belongs to $\mathfrak{I}$. For any distinguished $n$-exangle $A\longrightarrow X_{1}\longrightarrow X_2 \longrightarrow \cdots \longrightarrow X_{n}\longrightarrow C \stackrel{\delta}\dashrightarrow$, we have that $\mathbb{E}(i,A)(\delta)\in {\bf PB}(\mathfrak{I})$.
    For any $j\in\mathfrak{I}^{\perp}$, then by the discussion above we have $j\in{\bf PB}(\mathfrak{I})$-{\bf inj}. Hence $\mathbb{E}(i,j)(\delta)=0$, this shows that $i\in {^{\perp}(\mathfrak{I}^{\perp})}$. Since by Theorem \ref{Thm2}, ${^{\perp}(\mathfrak{I}^{\perp})}=\mathfrak{I}$, then $i\in\mathfrak{I}$, as we desired.
  \end{proof}
\end{proposition}

\begin{corollary}\label{Coro1}
Let $\mathfrak{I}$ be a special precovering ideal of $\mathscr{C}$ and take $\mathbb{F}={\bf PB}(\mathfrak{I})$, then $({\bf Ph}(\mathbb{F}),\mathbb{F}$-${\bf inj})$ is an $n$-ideal cotorsion pair.
\begin{proof}
By Proposition \ref{Prop5}, $\mathbb{F}$-${\bf inj}=\mathfrak{I}^{\perp}=({\bf Ph}(\mathbb{F}))^{\perp}$ and ${\bf Ph}(\mathbb{F})$ is a special precovering ideal. So by Theorem \ref{Thm2}, $({\bf Ph}(\mathbb{F}),({\bf Ph}(\mathbb{F}))^{\perp})$ is an $n$-ideal cotorsion pair. Hence $({\bf Ph}(\mathbb{F}),\mathbb{F}$-${\bf inj})$ is an $n$-ideal cotorsion pair.
  \end{proof}
\end{corollary}

\begin{lemma}\label{lemma5}
  If $i_0\colon A'\rightarrow Y_1$ is an $\mathbb{F}$-inflation which factors through an $\mathfrak{s}$-inflation $i\colon A'\rightarrow X_1$, then $i$ is an $\mathbb{F}$-inflation.
  \begin{proof}
    By definition, we have the following diagram:
      $$\xymatrix{
        A' \ar@{=}[d] \ar[r]^-{i} & X_1 \ar[d]^-{g} \ar[r] & X_2 \ar[r] & \cdots  \ar[r] & X_n  \ar[r] & C'  \ar@{-->}^{\delta}[r] & \\
        A' \ar[r]^-{i_0} & Y_1 \ar[r] & Y_2 \ar[r] & \cdots \ar[r] & Y_n \ar[r] & C \ar@{-->}^-{\theta}[r] & .
      }$$
    By Lemma \ref{lemma1}, the above diagram can be completed as follows
      $$\xymatrix{
        A' \ar@{=}[d] \ar[r]^-{i} & X_1 \ar[d]^-{g} \ar[r] & X_2 \ar[r] \ar[d] & \cdots  \ar[r] & X_n \ar[d] \ar[r] & C' \ar[d]^-{h} \ar@{-->}^{\delta}[r] & \\
        A' \ar[r]^-{i_0} & Y_1 \ar[r] & Y_2 \ar[r] & \cdots \ar[r] & Y_n \ar[r] & C \ar@{-->}^-{\theta}[r] &
      }$$
    such that $\delta=h^*\theta$. Since $\theta\in\mathbb{F}(C,A')$, $\delta\in\mathbb{F}(C',A')$, we obtain that $i$ is an $\mathbb{F}$-inflation.
  \end{proof}
\end{lemma}

We say that an additive subfunctor $\mathbb{F}\subset \mathbb{E}$ has {\em enough injective morphisms} if for every object $A'\in\mathscr{C}$, there exists a distinguished $n$-$\mathbb{F}$-exangle
$$
  A'\stackrel{e}\longrightarrow X_{1}\longrightarrow X_2 \longrightarrow \cdots \longrightarrow X_{n}\longrightarrow C \stackrel{\delta}\dashrightarrow
$$
with $e\colon A'\rightarrow X_1\in\mathbb{F}$-${\bf inj}$. The notion of {\em enough projective morphisms} is defined dually.

\begin{definition}\label{F16}
  {\em An additive subfunctor of $\mathbb{F}\subseteq\mathbb{E}$ has enough special injective morphisms if for every object $A'\in\mathscr{C}$, there exists a distinguished $n$-$\mathbb{F}$-exangle obtained by an $n$-$\mathbb{F}$-phantom-morphism $f\colon C\rightarrow C'$ and an $n$-exangle $\delta$, namely, we have the following commutative diagram
    $$\xymatrix{
      A' \ar@{=}[d] \ar[r]^-{e} & X_1 \ar[d] \ar[r] & X_2 \ar[r] \ar[d] & \cdots  \ar[r] & X_n \ar[d] \ar[r] & C \ar[d]^-{f} \ar@{-->}^{f^*\delta}[r] & \\
      A' \ar[r] & Y_1 \ar[r] & Y_2 \ar[r] & \cdots \ar[r] & Y_n \ar[r] & C' \ar@{-->}^-{\delta}[r] &
    }$$
  with $e\colon A'\rightarrow X_1$ being in $\mathbb{F}$-{\bf inj}. The notion of enough special projective morphisms is defined dually.}
\end{definition}

\begin{proposition}\label{Prop6}
  Let $\mathbb{F}\subseteq\mathbb{E}$ be an additive subfunctor which has enough injective morphisms, then ${\bf Ph}(\mathbb{F})={^{\perp}{\mathbb{F}}}$-${\bf inj}$.
  \begin{proof}
    Noting that the pair $({\bf Ph}(\mathbb{F}), \mathbb{F}$-${\bf inj})$ is an $n$-orthogonal pair, for $i\in{\bf Ph}(\mathbb{F}),j\in\mathbb{F}$-${\bf inj}$ and $\delta\in\mathbb{E}(C,A')$, we have the commutative diagram
      $$\xymatrix@=0.8cm{
       & A' \ar[rr] \ar@{=}[dd] \ar[dl]^{j} & & Z_1 \ar[rr] \ar[dl] \ar[dd] & & \cdots \ar[rr] \ar@{..>}[dd] \ar@{..>}[dl] & & Z_n \ar[rr] \ar[dd] \ar[dl] & & C' \ar@{=}[dl] \ar[dd]^-{i} \ar@{-->}[rr]^{i^*\delta} & & \\
       A \ar[rr] \ar@{=}[dd] & & U_1 \ar[dd] \ar[rr] & & \cdots \ar[rr] \ar@{..>}[dd] & & U_n \ar[dd] \ar[rr] & & C' \ar@{-->}[rr] \ar[dd]^-{i} & & & \\
       & A' \ar[rr] \ar[dl]^-{j} & & X_1 \ar[rr] \ar[dl] & & \cdots \ar[rr] \ar@{..>}[dl] & & X_n \ar[rr] \ar[dl] & & C \ar@{=}[dl] \ar@{-->}[rr]^-{\delta} & & \\
       A \ar[rr] & & Y_1 \ar[rr] & & \cdots \ar[rr] & & Y_n \ar[rr] & & C \ar@{-->}[rr]^-{j_*\delta} & & & \\
      }$$
    Then we obtain that $\mathbb{E}(i,j)(\delta)=0$, and then ${\bf Ph}(\mathbb{F})\subseteq {^{\perp}{\mathbb{F}}}$-${\bf inj}$. Conversely, take any morphism $f\colon X\rightarrow C$ in ${^{\perp}{\mathbb{F}}}$-${\bf inj}$ and any $\eta\in\mathbb{E}(C,A)$, there exists an $\mathbb{F}$-inflation $e\colon A\rightarrow Y$ which belongs to $\mathbb{F}$-${\bf inj}$, since $\mathbb{F}$ has enough injective morphisms. Then we have the commutative diagram
      $$\xymatrix@=0.8cm{
        & A \ar[rr]^-{d_\delta^0} \ar@{=}[dd] \ar[dl]_-{e} & & W_1 \ar[rr] \ar@{-->}[dlll]^{g} \ar[dl] \ar[dd] & & \cdots \ar[rr] \ar@{..>}[dd] \ar@{..>}[dl] & & W_n \ar[rr] \ar[dd] \ar[dl] & & X \ar@{=}[dl] \ar[dd] \ar@{-->}[rr]^{f^*\eta} & & \\
        Y \ar[rr] \ar@{=}[dd] & & V_1 \ar[dd] \ar[rr] & & \cdots \ar[rr] \ar@{..>}[dd] & & V_n \ar[dd] \ar[rr] & & X \ar@{-->}[rr] \ar[dd]^-{f} & & & \\
        & A \ar[rr] \ar[dl]_-{e} & & Q_1 \ar[rr] \ar[dl] & & \cdots \ar[rr] \ar@{..>}[dl] & & Q_n \ar[rr] \ar[dl] & & C \ar@{=}[dl] \ar@{-->}[rr]^-{\eta} & & \\
        Y \ar[rr] & & P_1 \ar[rr] & & \cdots \ar[rr] & & P_n \ar[rr] & & C \ar@{-->}[rr]^-{e_*\eta} & & &
       }$$
    Since $f\in {^{\perp}{\mathbb{F}}}$-${\bf inj}$ and $e\in {\mathbb{F}}$-${\bf inj}$, we get that $\mathbb{E}(f,e)=0$, i.e., $e_*f^*\delta=0$. Thus, there exists $g\colon W_1\rightarrow Y$ such that $gd_\delta^0=e$. By Lemma \ref{lemma5}, we obtain that $d_\delta^0$ is an $\mathbb{F}$-inflation since $e$ is an $\mathbb{F}$-inflation and $d_\delta^0$ is an $\mathfrak{s}$-inflation. Hence, $f^*\eta\in\mathbb{F}(X,A)$, i.e., $f\in{\bf Ph}(\mathbb{F})$.
  \end{proof}
\end{proposition}

\begin{corollary}\label{Coro2}
  Let $\mathfrak{I}$ be an ideal of $\mathscr{C}$ such that $(\mathfrak{I},\mathfrak{I}^{\perp})$ is a complete $n$-ideal cotorsion pair. Take $\mathbb{F}={\bf PB}(\mathfrak{I})$, then $\mathbb{F}\subseteq \mathbb{E}$ is an additive subfunctor which has enough special injective morphisms and $\mathfrak{I}={\bf Ph}(\mathbb{F})$.
  \begin{proof}
    By Proposition \ref{Prop4}, $\mathbb{F}$ is an additive subfunctor of $\mathbb{E}$. Since $(\mathfrak{I},\mathfrak{I}^{\perp})$ is a complete $n$-ideal cotorsion pair and $\mathbb{F}$-${\bf inj}=\mathfrak{I}^{\perp}$, $\mathbb{F}$-{\bf inj} is a special preenveloping ideal. That is, each $A\in\mathscr{C}$ has a special $\mathbb{F}$-{\bf inj}-preenvelope $j\colon A\rightarrow X_1$. Thus, there exists a distinguished $n$-exangle $A\longrightarrow Y_{1}\longrightarrow Y_{2}\longrightarrow \cdots \longrightarrow  Y_{n}\longrightarrow C \stackrel{\delta}\dashrightarrow$ such that we have the following commutative diagram
      $$\xymatrix{
      A \ar@{=}[d] \ar[r]^-{j} & X_1 \ar[d] \ar[r] & X_2 \ar[d] \ar[r] & \cdots  \ar[r] & X_n \ar[d] \ar[r] & C' \ar[d]^-{i} \ar@{-->}^{i^*\delta}[r] & \\
      A \ar[r] & Y_1 \ar[r] & Y_2 \ar[r] & \cdots \ar[r] & Y_n \ar[r] & C \ar@{-->}^{\delta}[r] &
      }$$
    with $i\colon C'\rightarrow C$ being in ${^{\perp}{\mathbb{F}}}$-{\bf inj}. By the definition of $\mathbb{F}$-{\bf inj}-preenvelopes, $j\in \mathbb{F}$-{\bf inj}. Since $(\mathfrak{I},\mathfrak{I}^{\perp})$ is complete, $\mathfrak{I}$ is a special precovering ideal. By Proposition \ref{Prop5}, $\mathfrak{I}={\bf Ph}({\bf PB}(\mathfrak{I}))={\bf Ph}(\mathbb{F})$. Note that $i\in {^{\perp}{\mathbb{F}}}$-${\bf inj}=\mathfrak{I}={\bf Ph}(\mathbb{F})$. Therefore, we finish the proof.
  \end{proof}
\end{corollary}

\begin{corollary}\label{Coro3}
  If $\mathbb{F}\subseteq \mathbb{E}$ has enough special injective morphisms, then the ideal $\mathbb{F}$-{\bf inj} is a special preenveloping ideal.
  \begin{proof}
    Since $\mathbb{F}\subseteq \mathbb{E}$ has enough special injective morphisms, for any $A\in\mathscr{C}$ there exists a distinguished $n$-$\mathbb{F}$-exangle $A \longrightarrow X_{1}\longrightarrow X_2 \longrightarrow \cdots \longrightarrow X_{n}\longrightarrow C \stackrel{\delta}\dashrightarrow$ such that we have the following commutative diagram
      $$\xymatrix{
        A \ar@{=}[d] \ar[r]^-{e} & Y_1 \ar[d] \ar[r] & Y_2 \ar[d] \ar[r] & \cdots  \ar[r] & Y_n \ar[d] \ar[r] & C' \ar[d]^-{i} \ar@{-->}^{i^*\delta}[r] & \\
        A \ar[r] & X_1 \ar[r] & X_2 \ar[r] & \cdots \ar[r] & X_n \ar[r] & C \ar@{-->}^{\delta}[r] &
        }$$
    with  $i\in{\bf Ph}(\mathbb{F})$ and $e\in\mathbb{F}$-{\bf inj}. By Proposition \ref{Prop6}, we obtain that $i\in {^{\perp}{\mathbb{F}}}$-{\bf inj}. Thus, ${\mathbb{F}}$-{\bf inj} is a special preenveloping ideal.
  \end{proof}
\end{corollary}

\section{Special Precovering Ideals and Phantom Morphisms }
  In this section we study the connections between special precovering ideals and $n$-$\mathbb{F}$-phantom morphisms. As before, we still assume that $(\mathscr{C}, \mathbb{E},\mathfrak{s})$ is an $n$-exangulated category.

\begin{lemma}\label{lemma6}
  Let $\mathbb{F}\subseteq\mathbb{E}$ and $A\in\mathscr{C}$. Take a distinguished $n$-exangle $K\longrightarrow P_{1}\longrightarrow P_2 \longrightarrow \cdots \longrightarrow P_{n}\stackrel{p}\longrightarrow A \stackrel{\delta}\dashrightarrow$ satisfying $p\colon P_n\rightarrow A$ is a projective morphism. Then $\phi\colon X\rightarrow A$ is an $n$-$\mathbb{F}$-phantom morphism if and only if $\mathbb{E}(\phi,K)(\delta)\in\mathbb{F}(X,K)$.
  \begin{proof}
    By the definition of $n$-$\mathbb{F}$-phantom morphisms, the necessity is clear.  Let
    $$
      A' \longrightarrow X_{1}\longrightarrow X_2 \longrightarrow \cdots \longrightarrow X_{n}\longrightarrow A \stackrel{\eta}\dashrightarrow
    $$ be a distinguished $n$-exangle.
  We have the following commutative diagram
      $$\xymatrix{
        A' \ar@{=}[d] \ar[r] & Z_1 \ar[d] \ar[r] & Z_2 \ar[r] \ar[d] & \cdots \ar[r] & Z_n \ar[d] \ar[r] & P_n \ar@{-->}[dl]_{h^n} \ar[d]^{p} \ar@{-->}^{p^*\eta=0}[r] & \\
        A' \ar[r] & X_1 \ar[r] & X_2 \ar[r] & \cdots \ar[r] & X_n \ar[r]^{d_X^n} & A \ar@{-->}^-{\eta}[r] & .
      }$$
    Hence there exists $h^n\colon P_n\rightarrow X_n$ such that $d_X^nh^n=p$. Thus we have a commutative square
      $$\xymatrix{
        K \ar[r] & P_1 \ar[r] & P_2 \ar[r] & \cdots  \ar[r] & P_n \ar[d] \ar[r] & A \ar@{=}[d] \ar@{-->}^{\delta}[r] & \\
        A' \ar[r] & X_1 \ar[r] & X_2 \ar[r] & \cdots \ar[r] & X_n \ar[r] & A \ar@{-->}^{\eta}[r] & .
      }$$
    By the dual of Lemma \ref{lemma1}, the above diagram can be completed as follows
      $$\xymatrix{
        K \ar[r] \ar[d]^{g} & P_1 \ar[d] \ar[r] & P_2 \ar[d] \ar[r] & \cdots  \ar[r] & P_n \ar[d] \ar[r] & A \ar@{=}[d] \ar@{-->}^{\delta}[r] & \\
        A' \ar[r] & X_1 \ar[r] & X_2 \ar[r] & \cdots \ar[r] & X_n \ar[r] & A \ar@{-->}^{\eta}[r] & .
      }$$
    Then we have that $\mathbb{E}(\phi,A')(\eta)=\phi^*g_*\delta=g_*\phi^*\delta$. Since $\phi^*\delta\in\mathbb{F}(X,K)$, we have that $\mathbb{E}(\phi,A')(\eta)\in\mathbb{F}(X,A')$. Hence, $\phi$ is an $n$-$\mathbb{F}$-phantom morphism.
  \end{proof}
\end{lemma}

\begin{lemma}\label{lemma9}
  Let $\mathfrak{J}$ be an ideal of $\mathscr{C}$ and consider the following diagram of distinguished $n$-exangles
    $$\xymatrix{
      A \ar@{=}[d] \ar[r]^-{j} & X_1 \ar[d] \ar[r] & X_2 \ar[r] \ar[d] & \cdots  \ar[r] & X_n \ar[d] \ar[r] & C' \ar[d]^-{i} \ar@{-->}^{i^*\delta}[r] & \\
      A \ar[r] & Y_1 \ar[r] & Y_2 \ar[r] & \cdots \ar[r] & Y_n \ar[r] & C \ar@{-->}^-{\delta}[r] & ,
    }$$
  where $i\in{^{\perp}\mathfrak{J}}$. If $j\in\mathfrak{J}$, then $j$ is an $\mathfrak{J}$-preenvelope.
  \begin{proof}
    Consider any morphism $j'\colon A\rightarrow D$ in $\mathfrak{J}$. Since $\mathbb{E}(C',j')\circ\mathbb{E}(i,A)(\delta)=\mathbb{E}(i,j')(\delta)=0$, there exists $X_1\rightarrow D$ such that $j'$ factors through $j$. Hence $j$ is an $\mathfrak{J}$-preenvelope.
  \end{proof}
\end{lemma}

\begin{theorem}\label{Thm4}
  Suppose that $\mathscr{C}$ has enough projective morphisms and $\mathbb{F}\subseteq \mathbb{E}$ is an additive subfunctor with enough injective morphisms. Let $A\in\mathscr{C}$ and $K\longrightarrow P_{1}\longrightarrow P_2 \longrightarrow \cdots \longrightarrow P_{n}\stackrel{p}\longrightarrow A \stackrel{\delta}\dashrightarrow$ be a distinguished $n$-exangle with $p\colon P_n\rightarrow A$ being a projective morphism. For any $\mathbb{F}$-inflation $e$ with $e\in\mathbb{F}$-{\bf inj}, consider the following commutative diagram
    $$\xymatrix{
      K \ar[r] \ar[d]^{e} & P_1 \ar[d] \ar[r] & P_2 \ar[d] \ar[r] & \cdots  \ar[r] & P_n \ar[d] \ar[r] & A \ar@{=}[d] \ar@{-->}^{\delta}[r] & \\
      C \ar[r] & X_1 \ar[r] & X_2 \ar[r] & \cdots \ar[r] & X_n \ar[r]^{\phi} & A \ar@{-->}^{e_*\delta}[r] & ,
    }$$
   then $\phi\colon X_n\rightarrow A$ is a special ${\bf Ph}(\mathbb{F})$-precover of $A$.
  \begin{proof}
    It suffices to show that: (i) $e\in{\bf Ph}(\mathbb{F})^{\perp}$; (ii) $\phi$ is a special ${\bf Ph}(\mathbb{F})$-precover. Note that $\mathbb{F}$-{\bf inj}$\subseteq {\bf Ph}(\mathbb{F})^{\perp}$ since $({\bf Ph}(\mathbb{F}),\mathbb{F}$-${\bf inj})$ is an $n$-orthogonal pair. Hence, {\rm (i)} holds. For {\rm (ii)}, it suffices to prove that $\mathbb{E}(\phi,K)(\delta)\in\mathbb{F}(X_n,K)$ by Lemma \ref{lemma6} and the dual of Lemma \ref{lemma9}. Consider the following diagram
      $$\xymatrix@=1.2cm{
        K \ar@{=}[d] \ar[r] & Z_1 \ar[d] \ar[r] & Z_2 \ar[d] \ar[r] & \cdots \ar[r] & Z_n \ar[d] \ar[r] & X_n \ar[d]^-{\phi} \ar@{-->}^{\phi^*\delta}[r] & \\
        K \ar[d]^{e} \ar[r] & P_1 \ar[d] \ar[r] & P_2 \ar[d] \ar[r] & \cdots \ar[r] & P_n \ar[d] \ar[r] & A \ar@{=}[d] \ar@{-->}^{\delta}[r] & \\
        C  \ar[r] & X_1 \ar[r] & X_2 \ar[r] & \cdots \ar[r] & X_n  \ar[r] & A  \ar@{-->}^-{e_*\delta}[r] & .
        }$$
    By the compositions, we can get the following morphism of distinguished $n$-exangles
      $$\xymatrix@=0.8cm{
        K \ar[d]^{e} \ar[r]^{d_Z^0} & Z_1 \ar@{-->}[dl]|{h} \ar[d] \ar[r] & Z_2 \ar[d] \ar[r] & \cdots \ar[r] & Z_n \ar[d] \ar[r] & X_n \ar[dl]|{1_{X_n}} \ar[d]^{\phi} \ar@{-->}^{\phi^*\delta}[r] & \\
        C  \ar[r] & X_1 \ar[r] & X_2 \ar[r] & \cdots \ar[r] & X_n \ar[r]^{\phi} & A \ar@{-->}^-{e_*\delta}[r] & .
        }$$
    By Lemma \ref{lemma2}, there exists $h\colon Z_1\rightarrow C$ such that $hd_Z^n=e$. Since $e$ is an $\mathbb{F}$-inflation, by Lemma \ref{lemma5} we get that $d_Z^0$ is an $\mathbb{F}$-inflation. This shows that $\mathbb{E}(\phi,K)(\delta)\in\mathbb{F}(X_n,K)$.
  \end{proof}
\end{theorem}

By Theorem \ref{Thm4}, we can get the following results.

\begin{theorem}\label{Thm6}
  Let $(\mathscr{C}, \mathbb{E},\mathfrak{s})$ be an $n$-exangulated category with enough projective morphisms. Then the ideal $\mathfrak{I}$ of $\mathscr{C}$ is special precovering if there exists an additive subfunctor $\mathbb{F}\subseteq \mathbb{E}$ with enough injective morphisms such that $\mathfrak{I}={\bf Ph}(\mathbb{F})$.
  \begin{proof}
    It suffices to show that for any $A\in\mathscr{C}$, $A$ has a speical {\bf Ph}$(\mathbb{F})$-precover. By Theorem \ref{Thm4} we finish the proof.
  \end{proof}
\end{theorem}


\section{Salce's Lemma}
It is well known that Salce's Lemma plays an important role in the classical approximation theory (cf. \cite{S}). Throughout this section, let $(\mathscr{C}, \mathbb{E},\mathfrak{s})$ be an extriangulated category with the Condition (WIC) and $\mathcal{T}$ be a nicely embedded $n$-cluster tilting subcategory of an extriangulated category $\mathscr{C}$. In this section, we prove Salce's Lemma in $\mathcal{T}$.

\begin{lemma}\label{lemma10}
  Recall that a nicely embedded $n$-cluster tilting subcategory $\mathcal{T}$ of an extriangulated category $\mathscr{C}$ is an $n$-exangulated category. Given a morphism $f\in\mathcal{T}(C',C)$ and an $\mathfrak{s}$-decomposable distinguished $n$-exangle $A\longrightarrow Y_{1}\longrightarrow Y_{2}\longrightarrow \cdots \longrightarrow Y_{n}\longrightarrow C \stackrel{\rho}\dashrightarrow$ in $\mathcal{T}$, then $\mathbb{E}^n(f,A)(\rho)$ has an $\mathfrak{s}$-decomposable realization.
  \begin{proof}
    Consider the $\mathbb{E}$-triangle $N_n\rightarrow Y_n\rightarrow C\stackrel{\rho_{(n)}}\dashrightarrow$, then we have the following commutative diagram
  $$\xymatrix{
    N_n \ar@{=}[d] \ar[r] & Z_n \ar[d] \ar[r] & C' \ar[d]^-{f} \ar@{-->}[r]^-{f^*\rho_{(n)}} & \\
    N_n \ar[r] & Y_n \ar[r] & C \ar@{-->}[r]^-{\rho_{(n)}} & .
  }$$
By the definition of nicely embedded $n$-cluster tilting subcategories, there exists a right $\mathcal{T}$-approximation $\gamma^n\colon X_n\rightarrow Z_n$ of $Z_n$  and it is also an $\mathfrak{s}$-deflation. Hence there exists $N^{'}_n$ such that we have the $\mathbb{E}$-triangle $N^{'}_n\rightarrow X_n\rightarrow C'\dashrightarrow$, then by {\rm (ET3)} in extriangulated categories we get a morphism of $\mathbb{E}$-triangles
  $$\xymatrix{
    N^{'}_n \ar[d]_-{f'^n} \ar[r] & X_n \ar[d]^-{f^n} \ar[r] & C' \ar[d]^-{f} \ar@{-->}[r]^-{\delta_{(n)}} & \\
    N_n \ar[r] & Y_n \ar[r] & C \ar@{-->}[r]^-{\rho_{(n)}} &
    }$$
with $X_n\in\mathcal{T}$. Then we consider the next $\mathbb{E}$-triangle $N_{n-1}\rightarrow Y_{n-1}\rightarrow N_n\stackrel{\rho_{(n-1)}}\dashrightarrow$ and $f'^n\colon N^{'}_n \rightarrow N_n$. Similarly, there exists a right $\mathcal{T}$-approximation $X_{n-1}\rightarrow Z_{n-1}$ of $Z_{n-1}$ and a morphism of $\mathbb{E}$-triangles, which can be depicted by the following diagrams
  $$\xymatrix{
    N_{n-1} \ar@{=}[d] \ar[r] & Z_{n-1} \ar[d] \ar[r] & N'_n \ar[d]^{f'^n} \ar@{-->}[r] & \\
    N_{n-1} \ar[r] & Y_{n-1} \ar[r] & N_n \ar@{-->}[r]^{\rho_{(n-1)}} &
  }
$$
$$
\xymatrix{
    N^{'}_{n-1} \ar[d]_-{f'^{n-1}} \ar[r] & X_{n-1} \ar[d]^-{f^{n-1}} \ar[r] & N^{'}_n \ar[d]^-{f'^n} \ar@{-->}[r]^-{\delta_{(n-1)}} & \\
    N_{n-1} \ar[r] & Y_{n-1} \ar[r] & N_n \ar@{-->}[r]^-{\rho_{(n-1)}} &
    }$$
with $X_{n-1}\in\mathcal{T}$. Repeating this procedure and setting $N'_0=N_0=A$, $N'_{n+1}=C'$ and $N_{n+1}=C$, we obtain the following commutative diagram
$$\xymatrix@=0.6cm{
   & & & & & N'_2 \ar[dr] \ar[dd]^-{f'^2} & & & & & & \\
  N'_0 \ar@{=}[dd] \ar[rr]  & & X_1 \ar@{-->}[rr] \ar[dd]^-{f^1} \ar[dr] & & X_2 \ar@{-->}[rr] \ar[dd]^-{f^2} \ar[ur] & & \ddots \ar@{.>}[dd] \ar@{-->}[rr] \ar[dr] & & X_n \ar[dd]^-{f^n} \ar[rr] & & N'_{n+1} \ar[dd]^-{f} \\
   & & & N'_1 \ar[dd]^-{f'^1} \ar[ur] & & N_2 \ar[dr] & & N'_n \ar[ur] \ar[dd]^-{f'^n} & & \\
  N_0 \ar[rr] & & Y_1 \ar[rr] \ar[dr] & & Y_2 \ar[rr] \ar[ur] & & \ddots \ar[dr] \ar[rr] & & Y_n \ar[rr] & & N_{n+1} \\
  & & & N_1 \ar[ur] & & & & N_n \ar[ur] & & \\
}$$
with $\mathbb{E}$-triangles $N'_{i-1}\longrightarrow X_i \longrightarrow N'_i \stackrel{\delta_{(i-1)}}\dashrightarrow$ and $N_{i-1}\longrightarrow X_i \longrightarrow N_i \stackrel{\rho_{(i-1)}}\dashrightarrow$ for $1\leq i \leq n+1$ and $X_j\in\mathcal{T}$ for $1\leq j\leq n$. By \cite[Propsition 3.18 and 3.19]{HLN2}, we obtain
\begin{flalign*}
  \delta_{(1)}\cup\delta_{(2)}\cup\cdots\cup\delta_{(n)}& =((f'^1)^*\rho_{(1)}\cup\delta_{(2)})\cup\cdots\cup\delta_{(n)}\\
   & =\rho_{(1)}\cup((f'^1)_*\delta_{(2)})\cup\cdots\cup\delta_{(n)}\\
   & =\rho_{(1)}\cup((f'^2)^*\rho_{(2)})\cup\cdots\cup\delta_{(n)}\\
   & \cdots\\
   & =(\rho_{(1)}\cup\cdots\cup\rho_{(n-1)})\cup(f^*\delta_{(n)})\\
   & =\mathbb{E}^n(f,A)(\rho_{(1)}\cup\cdots\cup\rho_{(n)})\\
   & =\mathbb{E}^n(f,A)(\rho).
\end{flalign*}
Hence, the above construction gives an $\mathfrak{s}$-decomposable realization for $\mathbb{E}^n(f,A)(\delta)$.
  \end{proof}
\end{lemma}

We need the following notions.

\begin{definition}\label{F17}
  {\em Suppose that $(\mathscr{A},\mathbb{E},\mathfrak{s})$ is an $n$-exangulated category. An object $E\in\mathscr{A}$ is called $\mathfrak{s}$-injective if, for any distinguished $n$-exangle
  $$
    A\longrightarrow X_{1}\longrightarrow X_2 \longrightarrow \cdots \longrightarrow X_{n}\longrightarrow C \stackrel{\delta}\dashrightarrow,
  $$
  the induced morphism $\mathscr{A}(X_{1},E) \longrightarrow \mathscr{A}(A,E)$ is an epimorphism, or equivalently, $\mathbb{E}(C,E)=0$ for any $C\in\mathscr{A}$. An $n$-exangulated category has enough $\mathfrak{s}$-injective objects if for any $X\in\mathscr{A}$, there exists a distinguished $n$-exangle
  $$
    X\longrightarrow E_{1}\longrightarrow E_2 \longrightarrow \cdots \longrightarrow E_{n}\longrightarrow X' \stackrel{\delta}\dashrightarrow
  $$
  with $E_1,\cdots,E_n$ being $\mathfrak{s}$-injective objects.}
\end{definition}

\begin{definition}\label{F18}
  {\em Let $(\mathscr{A},\mathbb{E},\mathfrak{s})$ be an $n$-exangulated category and $\mathfrak{I}$ be an ideal of $\mathscr{A}$. Let
    $$\xymatrix{
      X_0 \ar[d]_{f^0} \ar[r] & X_1 \ar[d]_{f^1} \ar[r] & X_2 \ar[d]_{f^2} \ar[r] & \cdots  \ar[r] & X_n \ar[d]_{f^n} \ar[r] & X_{n+1} \ar[d]_{f^{n+1}} \ar@{-->}^{\delta}[r] & \\
      Y_0 \ar[r] & Y_1 \ar[r] & Y_2 \ar[r] & \cdots \ar[r] & Y_n \ar[r] & Y_{n+1} \ar@{-->}^{\eta}[r] &
      }$$
  be a morphism of distinguished $n$-exangles. We say $\mathfrak{I}$ is closed under $n$-extensions by $\mathfrak{s}$-injective objects if $f^0\in\mathfrak{I}$ and $X_{n+1}$ is an $\mathfrak{s}$-injective object, then one can deduce that $f^i\in\mathfrak{I}$, $i=1,2,\cdots,n$. Dually one can define that $\mathfrak{I}$ is closed under $n$-coextensions by $\mathfrak{s}$-projective objects.}
\end{definition}

Now we can state Salce's Lemma.
\begin{theorem}\label{Thm3}
  (Salce's Lemma) Let $(\mathcal{T},\mathbb{E}^n,\mathfrak{s}^n)$ be a nicely embedded $n$-cluster tilting subcategory of $(\mathscr{C},\mathbb{E},\mathfrak{s})$ which satisfies the Condition (WIC). Suppose that $(\mathfrak{I},\mathfrak{J})$ is an $n$-ideal cotorsion pair such that $\mathfrak{I}$ is closed under $n$-coextensions by $\mathfrak{s}^n$-projective objects and $\mathfrak{J}$ is closed under $n$-extensions by $\mathfrak{s}^n$-injective objects. If $\mathcal{T}$ has enough $\mathfrak{s}^n$-injective objects, then $\mathfrak{I}$ is a special precovering ideal if and only if $\mathfrak{J}$ is a special preenveloping ideal.
\end{theorem}

Before proving Theorem \ref{Thm3}, we need the following lemmas. For the convenience of readers, we give the proof of Lemma \ref{lemma7} {\rm\cite[Lemma 3.16]{LW}}.

\begin{lemma}
  \label{lemma7}
  Let $A\stackrel{f}{\longrightarrow} B\stackrel{g}{\longrightarrow} C\stackrel{\delta}{\dashrightarrow}$ and $A\stackrel{u}{\longrightarrow} D\stackrel{v}{\longrightarrow} E\stackrel{\eta}{\dashrightarrow}$ be two $\mathbb{E}$-triangles in $\mathscr{C}$. If there exists the following commutative diagram
    $$\xymatrix{
      A\ar@{=}[d] \ar[r]^{f} & B \ar[d]_{s} \ar[r]^{g} & C \ar[d]_{t}\ar@{-->}[r]^{\delta}&   \\
      A \ar[r]^{u} & D \ar[r]^{v} & E \ar@{-->}[r]^{\eta} &.
    }$$
Then the right square is a weak pullback, namely, if we have the morphisms $i\colon M\longrightarrow C$ and $j\colon M\longrightarrow D$ such that $ti=vj$, then there exists $l\colon M\rightarrow B$ such that $sl=j$ and $gl=i$.
  \begin{proof}
By \cite[Lemma  3.2]{NP}, we have that $v^{\ast}\eta=0$ . Using
  $$
    \delta_{\sharp}(\delta)=i^{\ast}\delta=i^{\ast}t^{\ast}\eta=(ti)^{\ast}\eta=(vj)^{\ast}\eta=j^{\ast}v^{\ast}\eta=0
  $$
  and the exactness of
  $$
    \mathscr{C}(M,B)\longrightarrow\mathscr{C}(M,C)\longrightarrow\mathbb{E}(M,A),
  $$
  we obtain that there exists $\rho:M\longrightarrow B$ such that $i=g\rho$. Moreover, by
  $$
    v(s\rho-j)=vs\rho-vj=tg\rho-vj=ti-vj=0
  $$
  and the exactness of
  $$
    \mathscr{C}(M,A)\longrightarrow\mathscr{C}(M,D)\longrightarrow\mathscr{C}(M,E),
  $$
  there exists $\beta:M\longrightarrow A$ such that $u\beta=s\rho-j$. Hence, we have that
  $$
    j=s\rho-u\beta=s\rho-sf\beta=s(\rho-f\beta)~\text{and}~g(\rho-f\beta)=g\rho-gf\beta=i.
  $$
  Taking $l=\rho-f\beta$, we have the following commutative diagram
    $$\xymatrix{
    M \ar@/_/[ddr]_{j} \ar@/^/[drr]^{i}
      \ar@{.>}[dr]|-{l}                   \\
     & B \ar[d]^{s} \ar[r]_{g} & C \ar[d]_{t}    \\
     & D \ar[r]^{v}     & E.
     }$$
  \end{proof}
\end{lemma}


\begin{lemma}\label{lemma8}
  Consider the following commutative diagram
    $$\xymatrix@=0.4cm{
      &  & &  & &  & &  & &  & & U_{n+1} \ar[dl]^-{f^{n+1}} \ar[dd]^-{g^{n+1}}  & & \\
       & &  & &  & &  & &  & & V_{n+1}  \ar[dd]^{h^{n+1}} & & & \\
      & X_0 \ar[rr] \ar[dl]_-{f^0} & & X_1 \ar[rr] \ar[dl] & & X_2 \ar[rr] \ar[dl] & & \cdots \ar[rr] \ar@{..>}[dl] & & X_n \ar[rr] \ar[dl] & & X_{n+1} \ar[dl] \ar@{-->}[rr]^-{\eta} & & \\
      Y_0 \ar[rr] & & Y_1 \ar[rr] & & Y_2 \ar[rr] & & \cdots \ar[rr] & & Y_n \ar[rr] & & Y_{n+1} \ar@{-->}[rr]^-{\delta} & & &
     }$$
  where $X_0\longrightarrow X_{1}\longrightarrow X_2 \longrightarrow \cdots \longrightarrow X_{n}\longrightarrow X_{n+1} \stackrel{\eta}\dashrightarrow$ and $Y_0\longrightarrow Y_{1}\longrightarrow Y_2 \longrightarrow \cdots \longrightarrow Y_{n}\longrightarrow Y_{n+1} \stackrel{\delta}\dashrightarrow$ are $\mathfrak{s}^n$-decomposable objects and $U_{n+1},V_{n+1}\in\mathcal{T}$. Then it can be completed as follows
   $$\xymatrix@=0.4cm{
      & X_0 \ar[rr]^{d_U^0} \ar@{=}[dd] \ar[dl]_{f^0} & & U_1 \ar[rr] \ar[dl] \ar[dd] & & U_2 \ar[rr] \ar[dl] \ar[dd] & & \cdots \ar[rr] \ar@{..>}[dd] \ar@{..>}[dl] & & U_n \ar[rr]^{d_U^n} \ar[dd]^{g^n} \ar[dl] & & U_{n+1} \ar[dl]^{f^{n+1}} \ar[dd]^-{g^{n+1}} \ar@{-->}[rr] & & \\
      Y_0 \ar[rr] \ar@{=}[dd] & & V_1 \ar[dd] \ar[rr] & & V_2 \ar[dd] \ar[rr] & & \cdots \ar[rr] \ar@{..>}[dd] & & V_n \ar[dd] \ar[rr] & & V_{n+1} \ar@{-->}[rr] \ar[dd]^-{h^{n+1}} & & & \\
      & X_0 \ar[rr] \ar[dl]_-{f^0} & & X_1 \ar[rr] \ar[dl] & & X_2 \ar[rr] \ar[dl] & & \cdots \ar[rr] \ar@{..>}[dl] & & X_n \ar[rr] \ar[dl] & & X_{n+1} \ar[dl]^{l^{n+1}} \ar@{-->}[rr]^-{\eta} & & \\
      Y_0 \ar[rr] & & Y_1 \ar[rr] & & Y_2 \ar[rr] & & \cdots \ar[rr] & & Y_n \ar[rr] & & Y_{n+1} \ar@{-->}[rr]^-{\delta} & & &
     }$$

\begin{proof}
    For convenience, we use the notations to denote the following morphisms:
    \begin{flalign*}
      d_U^i&\colon U_i\rightarrow U_{i+1}  &  d_V^i&\colon V_i\rightarrow V_{i+1}  & d_X^i&\colon X_i\rightarrow X_{i+1}  & d_Y^i&\colon Y_i\rightarrow Y_{i+1}\\
      g^{i}&\colon U_{i}\rightarrow X_{i}  & f^i&\colon U_i\rightarrow V_i  & l^i&\colon X_i\rightarrow Y_i  & h^i&\colon V_i\rightarrow Y_i.
    \end{flalign*}
    By Lemma \ref{lemma10}, we have the distinguished $n$-exangle $X_0 \longrightarrow U_1 \longrightarrow U_2 \longrightarrow \cdots \longrightarrow U_n \stackrel{d_U^n}\longrightarrow U_{n+1} \stackrel{(g_{n+1})^*\eta}\dashrightarrow$ with $U_1,U_2,\cdots,U_{n}\in\mathcal{T}$. It suffices to show that there exists morphisms $U_i\rightarrow V_i$ such that the following square is commutative for $1\leq i\leq n$
    $$\xymatrix@=0.4cm{
      U_i \ar[dd] \ar[rr] \ar[dr] & & U_{i+1} \ar[dd] \ar[dr]    \\
       & V_i \ar[dd] \ar[rr] & & V_{i+1} \ar[dd]  \\
       X_i \ar[dr] \ar[rr] & & X_{i+1} \ar[dr] & \\
       & Y_i \ar[rr] & & Y_{i+1}.
    }$$
    
    We show by descending induction on $i$. 
    
    Consider the first $\mathbb{E}$-triangle in the $\mathfrak{s}$-decomposable $n$-exangle 
    $$
      A\longrightarrow X_{1}\longrightarrow X_2 \longrightarrow \cdots \longrightarrow X_{n}\longrightarrow C \stackrel{\delta}\dashrightarrow.
    $$
    By Lemma \ref{lemma10}, we have the following commutative diagram
      $$\xymatrix{
         & V_{n} \ar[d]^{v^n} \ar[r]^{d_V^n} & V_{n+1} \ar@{=}[d] \\
      N_n \ar@{=}[d] \ar[r] & Z_n \ar[d] \ar[r]^{d_Z^n} & V_{n+1} \ar[d]^-{h^{n+1}} \ar@{-->}[r] & \\
        N_n \ar[r] & Y_n \ar[r]^{d_Y^n} & Y_{n+1} \ar@{-->}[r]^-{\delta_{(n)}} & ,
      }$$
    where $v^n$ is a right $\mathcal{T}$-approximation of $Z_{n}$ and moreover it is an $\mathfrak{s}$-deflation. By Lemma \ref{lemma8}, the bottom right square is a weak pullback. Noting that $h^{n+1}f^{n+1}d_U^n=l^{n+1}g^{n+1}d_U^n=l^{n+1}d_X^ng^n=d_Y^nl^ng^n$, we get that there is a morphism $t_n\colon U_n\rightarrow Z_n$ such that the following diagram is commutative
    \begin{flalign} \label{diyilahui}
     \xymatrix{
      U_n \ar@/_/[ddr]_{l^ng^n} \ar@/^/[drr]^{f^{n+1}d_U^n}
          \ar@{.>}[dr]|{t^n}    \\
       & Z_n \ar[d]^{s^n} \ar[r]_{d_Z^n} & V_{n+1} \ar[d]^{h^{n+1}}  \\
       & Y_n \ar[r]^{d_Y^n}     & Y_{n+1}.}\end{flalign}
    Since $v^n$ is a right $\mathcal{T}$-approximation of $Z_{n}$, there exists $f^n\colon U_n\rightarrow V_n$ such that $v^nf^n=t^n$. Hence we can replace $Z_n$ with $V_n$ in the above diagram and get a new commutative diagram
      $$\xymatrix@=0.3cm{
      U_n \ar[dd]_{g^n} \ar[rr]^{d_U^n} \ar@{.>}[dr]|{f^n} & & U_{n+1} \ar[dr]^{f^{n+1}}    \\
       & V_n \ar[dd]^{s^nv^n} \ar[rr]_{d_V^n} & & V_{n+1} \ar[dd]_{h^{n+1}}  \\
       X_n \ar[dr]^{l^n} & & & \\
       & Y_n \ar[rr]^{d_Y^n}  &  & Y_{n+1}.
       }$$
    So we complete the first step. For $1\leq i \leq n-1$, we have the following commutative diagram
      $$\xymatrix{
        N_{i-1} \ar@{=}[d] \ar[r] & Z_{i} \ar[d] \ar[r] & M_i \ar[d]^{{h'}^i} \ar@{-->}[r] & \\
        N_{i-1} \ar[r] & Y_{i} \ar[r] & N_i \ar@{-->}[r]^{\delta_{(i)}} & .
      }$$
    Assume that we have obtained the morphism $f^{i+1}\colon U_{i+1}\rightarrow V_{i+1}$.
    Since $M_i \longrightarrow V_{i+1} \longrightarrow M_{i+1} \dashrightarrow$ is an $\mathbb{E}$-triangle and $d_V^{i+1}f^{i+1}d_U^i=0$, there exists $\tau^{i}$ such that $\phi^i\tau^i=f^{i+1}d_U^i$. Then by hypothesis we have the following commutative diagram:
    \begin{equation}\label{fig3}
        \xymatrix@=0.4cm{
       & & U_i \ar[ddd] \ar[rr]^{d_U^i} \ar@{-->}[ddl]_{\tau^i} & & U_{i+1} \ar[rrr]^{d_U^{i+1}} \ar[dl]^{f^{i+1}} \ar[ddd] & & & U_{i+2} \ar[ddd]^{g^{i+2}} \ar[dl]^{f^{i+2}} \\
      Z_i \ar[rrr] \ar[dr] \ar[ddd] & & & V_{i+1} \ar[ddd]^{h^{i+1}} \ar[rrr]^{d_V^{i+1}} \ar[drr] & & & V_{i+2} \ar[ddd]^{h^{i+2}} \\
      & M_i \ar[rru]^{\phi^i} \ar[ddd]^-{h'^i} & & & & M_{i+1} \ar[ru] \ar[ddd]^{h'^{i+1}} \\
       & & X_i \ar[dll] \ar[rr] & & X_{i+1} \ar[dl]^{l^{i+1}} \ar[rrr] & & & X_{i+2} \ar[dl] \\
      Y_i \ar[dr]_{\xi^i} \ar[rrr] & & & Y_{i+1} \ar[drr] \ar[rrr]^{d_Y^{i+1}} & & & Y_{i+2} & \\
       & N_i \ar[rru]_{\psi^i} & & & & N_{i+1} \ar[ru] & &
    }\end{equation}
    By diagram chasing, we have that $\psi^i{h'}^{i}\tau^i=\psi^i\xi^il^ig^i$. Since $U_i\in\mathcal{T}$, by Lemma \cite[Lemma 3.33]{HLN2}, $\mathbb{E}^k(U_i,N_{i+1})=0$ for $1\leq k\neq  n-i-1 <n$. In particular, $\mathbb{E}^{n-1}(U_i,N_{i+1})=0$.
     Then by \cite[Proposition 3.27]{HLN2} we obtain that $\mathscr{C}(U_i,N_i)\longrightarrow \mathscr{C}(U_i,Y_{i+1})$ is injective. Hence ${h'}^{i}\tau^i=\xi^il^ig^i$. Similar to the diagram (\ref{diyilahui}), the weak pullback provides a morphism $t^i\colon U_i \rightarrow Z_i$ such that the following diagram is commutative
      $$\xymatrix@=0.6cm{
        U_i \ar@/_/[ddr]_{l^ig^i} \ar@/^/[drr]^{\tau^i}
            \ar@{.>}[dr]|{t^i}    \\
         & Z_i \ar[d]^{s^i} \ar[r]_{\eta^i} & M_i \ar[d]^{{h'}^{i}}  \\
         & Y_i \ar[r]^{\xi^i}     & N_i.
         }$$
    Since $v^i$ is a right $\mathcal{T}$-approximation of $Z_{i}$, there exists $f^i\colon U_i\rightarrow V_i$ such that $v^ir^i=t^i$. Hence we get a new diagram
    \begin{equation}\label{fig2}
      \xymatrix@=0.5cm{
      U_i \ar[dd]_{g^i} \ar@{=}[rr] \ar@{.>}[dr]|{f^i} & & U_{i} \ar[dr]^{\tau^i}    \\
       & V_i \ar[dd]^{s^iv^i} \ar[rr]_{\eta^i} & & M_i \ar[dd]_{{h'}^{i}}  \\
       X_i \ar[dr]^{l^i} & & & \\
       & Y_i \ar[rr]^{\xi^i}  &  & N_i.
       }
    \end{equation}
    Combining diagrams (\ref{fig2}) and (\ref{fig3}), we obtain the following commutative diagram
      $$\xymatrix{
        U_i \ar[d] \ar[r] & U_{i+1} \ar[d] \\
        V_i \ar[r] & V_{i+1}.
      }$$
    Then we finish the proof.
  \end{proof}
\end{lemma}
Now we are in the position to prove Theorem \ref{Thm3}.
\begin{proof}[\bf(Proof of Theorem \ref{Thm3})]
  Just verify that if $\mathfrak{I}$ is a special precovering ideal, then $\mathfrak{J}$ is a special preenveloping ideal. Dually one can prove the converse statement. For $A\in\mathcal{T}$, since $\mathcal{T}$ has enough $\mathfrak{s}^n$-injective objects, there exists distinguished $n$-exangle $A\longrightarrow E_{1}\longrightarrow E_2 \longrightarrow \cdots \longrightarrow E_{n}\longrightarrow C \stackrel{\delta}\dashrightarrow$. Since $\mathfrak{I}$ is a special precovering ideal, there exists a distinguished $n$-exangle $Y\longrightarrow Z_{1}\longrightarrow Z_2 \longrightarrow \cdots \longrightarrow Z_{n}\longrightarrow C \stackrel{\eta}\dashrightarrow$ and $j'\colon Y\rightarrow C' \in \mathfrak{J}$ such that $i\colon X_n\rightarrow C$ is an $\mathfrak{I}$-precover and the following diagram commutes
\begin{equation}\label{fig1}
      \xymatrix{
    Y \ar[d]_-{j'} \ar[r] & Z_1 \ar[d] \ar[r] & Z_2 \ar[d] \ar[r] & \cdots  \ar[r] & Z_n \ar[d] \ar[r] & C \ar@{=}[d] \ar@{-->}^{\eta}[r] & \\
    C' \ar[r] & X_1 \ar[r] & X_2 \ar[r] & \cdots \ar[r] & X_n \ar[r]^-{i} & C \ar@{-->}^{{j'}_*\eta}[r] & .\\
  }
\end{equation}
Consider the following commutative diagram
  $$\xymatrix{
    A \ar@{=}[d] \ar[r]^-{j} & U_1 \ar[d] \ar[r] & U_2 \ar[r] \ar[d] & \cdots  \ar[r] & U_n \ar[d] \ar[r] & X_n \ar[d]^-{i} \ar@{-->}^{i^*\delta}[r] & \\
    A \ar[r] & E_1 \ar[r] & E_2 \ar[r] & \cdots \ar[r] & E_n \ar[r] & C \ar@{-->}^-{\delta}[r] & .
  }$$
By Lemma \ref{lemma9}, it suffices to show $j\in\mathfrak{J}$. Consider the following diagram:
  $$
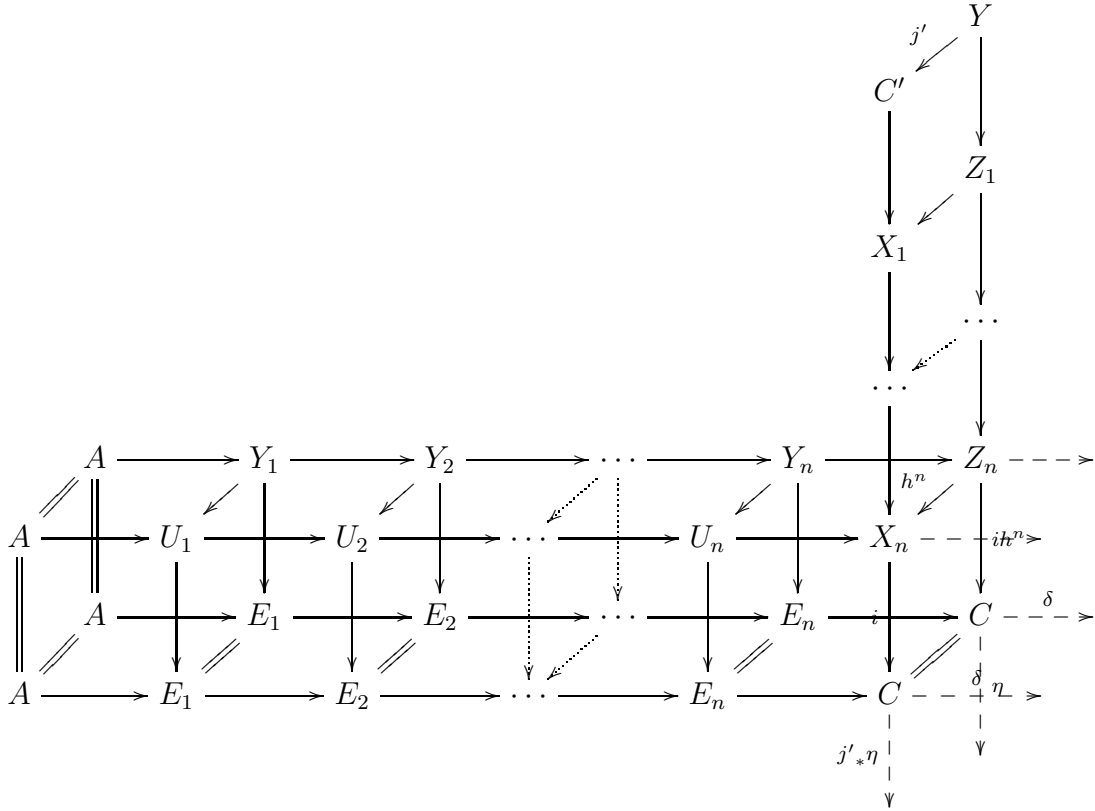
\xymatrix@=0.4cm{
    &  & &  & &  & &  & &  & & Y \ar[dl]_{j'} \ar[dd] & & \\
      & &  & &   & & & &  & & C' \ar[dd] & & & \\
    &   & &  & &  & &  & &  & & Z_1 \ar[dl] \ar[dd] & & \\
    & &  & &   & & & &  & & X_1 \ar[dd] & & & \\
    &   & &  & &  & &  & &  & & \cdots \ar@{.>}[dl] \ar[dd] & & \\
    & &  & &   & & & &  & & \cdots \ar[dd] & & & \\
    & A \ar[rr] \ar@{=}[dd] \ar@{=}[dl] & & Y_1 \ar[rr] \ar[dl] \ar[dd] & & Y_2 \ar[rr] \ar[dl] \ar[dd] & & \cdots \ar[rr] \ar@{..>}[dd] \ar@{..>}[dl] & & Y_n \ar[rr] \ar[dd] \ar[dl] & & Z_n \ar[dl]_{h^{n}} \ar[dd]^-{{ih^n}} \ar@{-->}[rr] & & \\
    A \ar[rr] \ar@{=}[dd] & & U_1 \ar[dd] \ar[rr] & & U_2 \ar[dd] \ar[rr] & & \cdots \ar[rr] \ar@{..>}[dd] & & U_n \ar[dd] \ar[rr] & & X_{n} \ar@{-->}[rr] \ar[dd]_-{i} & & & \\
    & A \ar[rr] \ar@{=}[dl] & & E_1 \ar[rr] \ar@{=}[dl] & & E_2 \ar[rr] \ar@{=}[dl] & & \cdots \ar[rr] \ar@{..>}[dl] & & E_n \ar[rr] \ar@{=}[dl] & & C \ar@{=}[dl] \ar@{-->}[rr]^-{\delta} \ar@{-->}[dd]^{\eta} & & \\
    A \ar[rr] & & E_1 \ar[rr] & & E_2 \ar[rr] & & \cdots \ar[rr] & & E_n \ar[rr] & & C \ar@{-->}[rr]^-{\delta} \ar@{-->}[dd]_{{j'}_*\eta} & & & \\
    & & & & & & & & & & & & & & \\
    & & & & & & & & & & & & & &
   }$$
By Lemma \ref{lemma8}, starting with diagram (\ref{fig1}) as the right two columns, we get the rows step by step, which are $n$-exangles except the top two rows.
In addtion, consider diagram (\ref{fig1}) as the bottom and the last square of the bottom two rows, we have the following commutative diagram
  $$
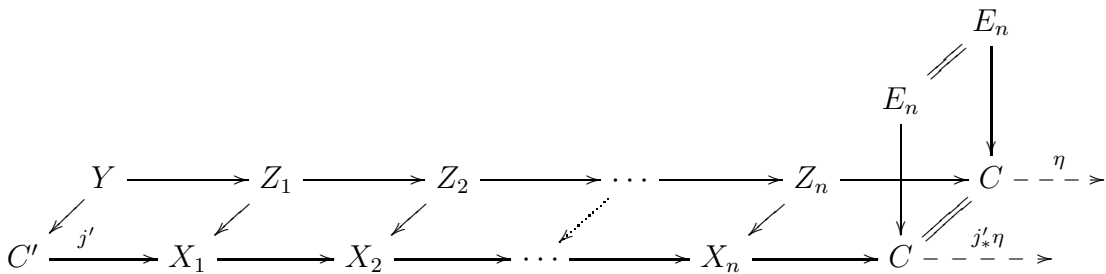
\xymatrix@=0.4cm{
    &  & &  & &  & &  & &  & & E_n \ar@{=}[dl] \ar[dd]  & & \\
     & &  & &  & &  & &  & & E_n  \ar[dd] & & & \\
    & Y \ar[rr] \ar[dl]^-{j'} & & Z_1 \ar[rr] \ar[dl] & & Z_2 \ar[rr] \ar[dl] & & \cdots \ar[rr] \ar@{..>}[dl] & & Z_n \ar[rr] \ar[dl] & & C \ar@{=}[dl] \ar@{-->}[rr]^-{\eta} & & \\
    C' \ar[rr] & & X_1 \ar[rr] & & X_2 \ar[rr] & & \cdots \ar[rr] & & X_n \ar[rr] & & C \ar@{-->}[rr]^-{j'_*\eta} & & &
   }$$
We can also complete it and get the columns step by step except the left two columns. By Lemma \ref{lemma8}, the two procedures give the same commutative diagram
  $$\xymatrix@=0.4cm{\
  & A \ar[rr] \ar@{=}[dd] \ar@{=}[dl] & & Y \ar@{=}[rr] \ar[dl] \ar[dd] & & Y \ar@{=}[rr] \ar[dl] \ar[dd] & & \cdots \ar@{=}[rr] \ar@{..>}[dd] \ar@{..>}[dl] & & Y \ar@{=}[rr] \ar[dd] \ar[dl] & & Y \ar[dl] \ar[dd] \ar@{-->}[rr] & & \\
  A \ar[rr] \ar@{=}[dd] & & C' \ar[dd] \ar@{=}[rr] & & C' \ar[dd] \ar@{=}[rr] & & \cdots \ar@{=}[rr] \ar@{..>}[dd] & & C' \ar[dd] \ar@{=}[rr] & & C' \ar@{-->}[rr] \ar[dd] & & & \\
  & A \ar[rr] \ar@{=}[ddd] \ar@{=}[dl] & & Y_1^1 \ar[rr] \ar[dl] \ar[ddd] & & Y_2^1 \ar[rr] \ar[dl] \ar[ddd] & & \cdots \ar[rr] \ar@{..>}[ddd] \ar@{..>}[dl] & & Y_n^1 \ar[rr] \ar[ddd] \ar[dl] & & Z_1 \ar[dl] \ar[ddd] \ar@{-->}[rr] & & \\
  A \ar[rr] \ar@{=}[ddd] & & U_1^1 \ar[ddd] \ar[rr] & & U_2^1 \ar[ddd] \ar[rr] & & \cdots \ar[rr] \ar@{..>}[ddd] & & U_n^1 \ar[ddd] \ar[rr] & & X_1 \ar@{-->}[rr] \ar[ddd] & & & \\
  & & & & & & & & & & & & & & \\
  & A \ar[rr] \ar@{=}[ddd] \ar@{=}[dl] & & \vdots \ar@{..>}[rr] \ar@{..>}[dl] \ar[ddd] & & \vdots \ar@{..>}[rr] \ar@{..>}[dl] \ar[ddd] & & \cdots \ar@{..>}[rr] \ar@{..>}[ddd] \ar@{..>}[dl] & & \vdots \ar@{..>}[rr] \ar[ddd] \ar@{..>}[dl] & & \vdots \ar@{..>}[dl] \ar[ddd] \ar@{-->}[rr] & & \\
  A \ar[rr] \ar@{=}[ddd] & & \vdots \ar[ddd] \ar@{..>}[rr] & & \vdots \ar[ddd] \ar@{..>}[rr] & & \cdots \ar@{..>}[rr] \ar@{..>}[ddd] & & \vdots \ar[ddd] \ar@{..>}[rr] & & \vdots \ar@{-->}[rr] \ar[ddd] & & & \\
  & & & & & & & & & & & & & & \\
  & A \ar[rr] \ar@{=}[dd] \ar@{=}[dl] & & Y_1 \ar[rr] \ar[dl]^{j^1} \ar[dd] & & Y_2 \ar[rr] \ar[dl] \ar[dd] & & \cdots \ar[rr] \ar@{..>}[dd] \ar@{..>}[dl] & & Y_n \ar[rr] \ar[dd] \ar[dl] & & Z_n \ar[dl]_{h^{n}} \ar[dd]^-{{ih^n}} \ar@{-->}[rr] & & \\
  A \ar[rr] \ar@{=}[dd] & & U_1 \ar[dd] \ar[rr] & & U_2 \ar[dd] \ar[rr] & & \cdots \ar[rr] \ar@{..>}[dd] & & U_n \ar[dd] \ar[rr] & & X_{n} \ar@{-->}[rr] \ar[dd]_-{i} & & & \\
  & A \ar[rr] \ar@{=}[dl] & & E_1 \ar[rr] \ar@{-->}[dd] \ar@{=}[dl] & & E_2 \ar[rr] \ar@{=}[dl] \ar@{-->}[dd] & & \cdots \ar[rr] \ar@{..>}[dl] \ar@{-->}[dd] & & E_n \ar[rr] \ar@{=}[dl] \ar@{-->}[dd] & & C \ar@{=}[dl] \ar@{-->}[rr]^-{\delta} \ar@{-->}[dd]^{\eta} & & \\
  A \ar[rr] & & E_1 \ar[rr] \ar@{-->}[dd] & & E_2 \ar[rr] \ar@{-->}[dd] & & \cdots \ar[rr] \ar@{-->}[dd] & & E_n \ar@{-->}[dd] \ar[rr] & & C \ar@{-->}[rr]^-{\delta} \ar@{-->}[dd]_{{j'}^*\eta} & & & \\
  & & & & & & & & & & & & & & \\
  & & & & & & & & & & & & & &
    }$$
We get the following morphism of distinguished $n$-exangles
  $$\xymatrix{
    Y \ar[d]_-{j'} \ar[r] & Y_1^1 \ar[d] \ar[r] & Y_1^2 \ar[d] \ar[r] & \cdots  \ar[r] & Y_1 \ar[d]^{j^1} \ar[r] & E_1 \ar@{=}[d] \ar@{-->}[r] & \\
    C' \ar[r] & U_1^1 \ar[r] & U_1^2 \ar[r] & \cdots \ar[r] & U_1 \ar[r] & E_1 \ar@{-->}[r] & .
  }$$
Since $\mathfrak{J}$ is closed under $n$-extensions by $\mathfrak{s}^n$-injective objects, it follows that $j^1\in\mathfrak{J}$. Since $j$ factors through $j^1$, we obtain that $j\in\mathfrak{J}$. Therefore, $\mathfrak{J}$ is a special preenveloping ideal. 
\end{proof}

In conclusion, we give the following theorem, which can be viewed as the higher version of \cite[Theorem 1]{FGHT} in nicely embedded $n$-cluster tilting subcategories of extriangulated categories.

\begin{theorem}\label{Thm5}
  Let $(\mathscr{C},\mathbb{E},\mathfrak{s})$ be an extriangulated category with the Condition (WIC) and $(\mathcal{T},\mathbb{E}^n,\mathfrak{s}^n)$ be a nicely embedded $n$-cluster tilting subcategory of $\mathscr{C}$ with enough $\mathfrak{s}^n$-injective objects and projective morphisms. Suppose that $\mathfrak{I}$ is an ideal of $\mathcal{T}$ such that $\mathfrak{I}^{\perp}$ is closed under $n$-extensions by $\mathfrak{s}^n$-injective objects. Then the following statements are equivalent:

  {\rm (i)} there exists an additive subfunctor $\mathbb{F}\subseteq\mathbb{E}^n$ with enough injective morphisms such that $\mathfrak{I}={\bf Ph}(\mathbb{F})$;

  {\rm (ii)} $\mathfrak{I}$ is a special precovering ideal;

  {\rm (iii)} the $n$-ideal cotorsion pair $(\mathfrak{I},\mathfrak{I}^{\perp})$ is complete;

  {\rm (iv)} the additive subfunctor ${\bf PB}(\mathfrak{I})\subseteq\mathbb{E}^n$ has enough special injective morphisms and $\mathfrak{I}={\bf Ph}({\bf PB}(\mathfrak{I}))$.
  \begin{proof}
    (i)$\Rightarrow$(ii) by Theorem \ref{Thm4}; (ii)$\Rightarrow$(iii) by Theorems \ref{Thm2} and \ref{Thm3}; (iii)$\Rightarrow$(iv) by Corollary \ref{Coro2}; (iv)$\Rightarrow$(i) is obvious.
  \end{proof}
\end{theorem}



\end{document}